\documentclass[11pt]{article}

\usepackage{times}
\usepackage{amsthm}
\usepackage{amsmath}
\usepackage{amssymb}
\usepackage{mathrsfs}
\usepackage{pb-diagram}
\usepackage{color}

\usepackage[colorlinks]{hyperref}

\def\o{\omega}
\def\si{\sigma}
\def\sig{\varsigma}

\def\O{\Omega}
\def\N{\mathbb{N}}
\def\T{\mathbb{T}}

\def\X{\mathcal X}
\def\Y{\mathcal Y}
\def\Z{\mathcal Z}

\def\H{\mathcal H}

\def\R{\mathbb{R}}

\def\S{\mathcal S}

\def\U{\mathcal U}

\def\e{{\sf e}}

\def\g{{\mathfrak g}}
\def\m{{\sf m}}
\def\wm{\widehat{\sf m}}

\def\z{\mathfrak z}

\def\({\left(}
\def\[{\left[}
\def\){\right)}
\def\]{\right]}

\def\bu{{\bullet}}

\def\si{\sigma}
\def\Si{\Sigma}

\def\G{{\sf G}}
\def\wG{\widehat{\sf{G}}}
\def\p{\parallel}
\def\<{\langle}
\def\>{\rangle}

\def\Op{\mathfrak{Op}}

\newtheorem{Theorem}{Theorem}[section]
\newtheorem{Remark}[Theorem]{Remark}
\newtheorem{Lemma}[Theorem]{Lemma}

\newtheorem{Corollary}[Theorem]{Corollary}
\newtheorem{Proposition}[Theorem]{Proposition}
\newtheorem{Definition}[Theorem]{Definition}

\numberwithin{equation}{section}

\begin{document}


\title{Quantizations on Nilpotent Lie Groups and Algebras\\ Having Flat Coadjoint Orbits}

\date{\today}

\author{M. M\u antoiu and M. Ruzhansky \footnote{
\newline
{M. M. has been supported by N\'ucleo Milenio de F\'isica Matem\'atica RC120002 and the Fondecyt Project 1160359. M. R. was supported in parts by the EPSRC
 grant EP/K039407/1 and by the Leverhulme Grant RPG-2014-02. 
}}
}

\date{\small}

\maketitle \vspace{-1cm}


\begin{abstract}
For a connected simply connected nilpotent Lie group $\G$ with Lie algebra $\g$ and unitary dual $\wG$ one has (a) a global quantization of operator-valued symbols defined on $\G\times\wG$\,, involving the representation theory of the group, (b) a quantization of scalar-valued symbols defined on $\G\times\g^*$, taking the group structure into account and (c) Weyl-type quantizations of all the coadjoint orbits $\big\{\O_\xi\mid\xi\in\wG\big\}$\,.
We show how these quantizations are connected, in the case when flat coadjoint orbits exist. This is done by a careful analysis of the composition of two different types of Fourier transformations. We also describe the concrete form of the operator-valued symbol quantization, by using Kirillov theory and the Euclidean version of the unitary dual and Plancherel measure. In the case of the Heisenberg group this corresponds to the known picture, presenting the representation theoretical pseudo-differential operators in terms of families of Weyl operators depending on a parameter. For illustration, we work out a couple of examples and put into evidence some specific features of the case of Lie algebras with one-dimensional center. When $\G$ is also graded, we make a short presentation of the symbol classes $S^m_{\rho,\delta}$\,, transferred from $\G\times\wG$ to $\G\times\g^*$ by means of the connection mentioned above.
\end{abstract}


\section{Introduction}\label{duci}

The article treats pseudo-differential operators associated to a connected simply connected nilpotent Lie group $\G$ with Haar measure $d\m(x)$ and Lie algebra $\g$\,. Denoting by $\g^*$ the dual of $\g$ and  by $\wG$ the unitary dual of $\G$\,, composed of unitary equivalence classses of irreducible representations,  one has various pseudo-differential calculi:
\begin{enumerate}
\item
A global quantization ${\sf Op}_{\G\times\wG}$ of operator-valued symbols \cite{FR,FR2,FR1,MR} defined on $\G\times\wG$\,, making strong use of the representation theory of the group (see \cite{RT, RT2} for the compact Lie group case). As an outcome, we get operators acting in function spaces defined on $\G$\,, as $L^2(\G)$ for instance.
\item
A quantization ${\sf Op}_{\G\times\g^*}$ of scalar-valued symbols defined on $\G\times\g^*$, taking the group law into account and different from the usual cotangent bundle quantization. Once again one gets operators acting in function spaces defined on $\G$\,.
\item
Quantizations ${\sf Ped}_\xi$ \cite{Pe}  of all the coadjoint orbits $\big\{\O_\xi\subset\g^*\!\mid\!\xi\in\wG\big\}$\,. This generalizes the Weyl calculus, seen as a pseudo-differential theory on the coadjoint orbits of the Heisenberg group. The emerging operators act on the Hilbert space $\H_\xi$ of the irreducible representation $\xi$ or on the corresponding space $\H_\xi^\infty$ of smooth vectors.
\end{enumerate}

To make things clear, let us indicate the basic formulae for ${\sf Op}_{\G\times\g^*}$  and ${\sf Op}_{\G\times\wG}$\,. 

\smallskip
Recall first that for connected simply connected nilpotent Lie groups the exponential map $\exp:\g\to\G$ is a global diffeomorphism; its inverse will be denoted by $\log$\,. Since $\g$ and $\g^*$ are dual finite-dimensional vector spaces, one can start with the familiar Kohn-Nirenberg  formula 
\begin{equation}\label{draci}
\big[\mathfrak{op}_{\g\times\g^*}({\sf f})\nu\big](X)=\int_\g\int_{\g^*}\!e^{i\<X-Y\mid \X\>}{\sf f}\!\left(X,\X\right)\nu(Y)\,dYd\X\,,
\end{equation}
associating to certain functions $\mathfrak f:\g\times\g^*\to\mathbb C$ operators acting on $\nu\in L^2(\g)$\,. Then, suitably composing with the functions $\exp$ and $\log$ at the level of the vectors or of the symbols, one easily arrives at the formula
\begin{equation}\label{trulcisor}
\big[{\sf op}_{\G\times\g^*}\!(f)u\big](x)=\int_\G\int_{\g^*}\!e^{i\<\log x-\log y\mid \X\>}f\big(x,\X\big)u(y)\,d\m(y)d\X
\end{equation}
in which $f:\G\times\g^*\to\mathbb C$ and $u:\G\to\mathbb C$ are suitable functions. Although $\G\times\g^*$ can be identified with the cotangent bundle of the manifold $\G$\,, it is clear that the group structure of $\G$ (or the Lie algebra structure of $\g$) plays no role and this makes \eqref{trulcisor} unsuitable for our purposes. In \eqref{draci}, for instance, if ${\sf f}$ only depends on $\X\in\g^*$, the operator $\mathfrak{op}_{\g\times\g^*}\!({\sf f})$ is just a convolution by the Euclidean Fourier transform of ${\sf f}$\,. The group involved in this convolution is just $(\g,+)$\,, leading to a commutative convolution calculus, and this is not what we want.

\smallskip
A better situation occurs if on $\g$\,, instead of the vector sum, one considers the group operation $(X,Y)\mapsto X\bullet Y:=\log[\exp(X)\exp(Y)]$ obtained from $\G$ by transport of structure; it is given by the Baker-Campbell-Hausdorff formula, consisting in our nilpotent case of a finite combination of successive commutators of $X$ and $Y$. We are lead to replace \eqref{draci} by
\begin{equation}\label{laci}
\big[\mathfrak{Op}_{\g\times\g^*}({\sf f})\nu\big](X)=\int_\g\int_{\g^*}\!e^{i\<(-Y)\bu X\mid \X\>}{\sf f}\!\left(X,\X\right)\nu(Y)\,dYd\X\,.
\end{equation}
Performing the same compositions with $\exp$ and $\log$\,, one finally gets instead of \eqref{trulcisor} the quantization formula
\begin{equation}\label{calcisor}
\big[{\sf Op}_{\G\times\g^*}\!(f)u\big](x)=\int_\G\int_{\g^*}\!e^{i\<\log(y^{-1}x)\mid \X\>}f\big(x,\X\big)u(y)\,d\m(y)\,d\X\,.
\end{equation}
To be precise, the needed transformations are $\,\nu\mapsto u:=\nu\circ\log$ and $\,{\sf f}\mapsto f:={\sf f}\circ\big(\log\otimes\,{\sf id}_{\g^*}\big)$\,, followed by a change of variable (the Haar measure on $\G$ corresponds to the Lebesgue measure on $\g$ under the exponential map). Now, if $f$ does not depend on $x\in\G$\,, one gets the non-commutative calculus of right-convolution operators on the group; see Remark \ref{aman}.

\smallskip
Clearly there is another strategy, using the expression $X\bu(-Y)$ in \eqref{laci} and (thus) the expression $\log(xy^{-1})$ in \eqref{calcisor}. They provide different (but similar) pseudo-differential calculi, reflecting once again the (eventual) non-commutatativity of our setting. Left convolutions are covered this way.

\smallskip
The global group quantization relies on the formula
\begin{equation}\label{ploicika}
\big[{\sf Op}_{\G\times\wG}(\si)u\big](x)=\int_\G\int_{\wG}{\rm Tr}_\xi\big[\xi(y^{-1}x)\si(x,\xi)\big]u(y)\,d\m(y)d\wm(\xi)\,,
\end{equation}
involving representation theoretical ingredients. Although the elements of $\wG$ are, by definition, classes of equivalence, we can treat them simply as irreducible representations chosen each in the corresponding class. We set $d\wm(\xi)$ for the Plancherel measure \cite{Di,Fo}. The symbol $\si$ can be seen as a family $\{\si(x,\xi)\mid (x,\xi)\in\G\times\wG\}$\,, where $\si(x,\xi)$ is an operator in the Hilbert space $\H_\xi$ of the irreducible representation $\xi$\,. Under various favorable circumstances the ingredients in \eqref{ploicika} make sense and define an operator acting on various spaces of functions or distributions on $\G$\,. The theory for graded nilpotent groups, as exposed in \cite{FR, FR1}, is already well developed, but many things can be said \cite{MR} even for unimodular type I locally compact groups.

\smallskip
In \cite[Sect.\;8]{MR} it has been shown that the two quantizations ${\sf Op}_{\G\times\wG}$ and ${\sf Op}_{\G\times\g^*}$ are actually equivalent in a strong sense, being both obtained from a crossed product $C^*$-algebraic construction, by applying two different (but both isomorphic) types of partial Fourier transforms in conjunction with the canonical Schr\"odinger representation of the crossed product $C^*$-algebra. The first Fourier transform is defined in terms of the duality between $\G$ and its unitary dual $\wG$ and relies on the non-commutative Plancherel theorem, valid also for large classes of non nilpotent groups. The second comes from identifying $\G$ with $\g$ by the exponential map, and then using the duality between $\g$ and $\g^*$. Composing one partial Fourier transform ${\sf id}\otimes\mathscr F_{\!(\G,\wG)}$ with the inverse ${\sf id}\otimes\mathscr F^{-1}_{\!(\G,\g^*)}$ of the other provides an isomorphism ${\sf id}\otimes\mathscr W$ between the two quantizations; they can be seen as two equivalent ways to represent operators on function spaces over a  connected simply connected nilpotent Lie group. Unfortunately, besides being an isomorphism, this composition seems in general rather hard to use being not very explicit.

\smallskip
However, for a subclass of nilpotent Lie groups, a remarkable fact occurs, that is the main subject of the present paper. It is known \cite{CG,Ki04} that the unitary dual $\wG$ can be understood via the coadjoint action of $\G$ on $\g^*$. Actually there is a one-to-one correspondence (homeomorphism) between classes of equivalences of irreducible representations and coadjoint orbits. Moreover, the representation is square integrable with respect to the center if and only if the corresponding coadjoint orbit is flat; we refer to \cite{CG,MW} and to our Subection \ref{hiperborean} for these notions.
A given nilpotent Lie group might not have any flat coadjoint orbit, but if it does, these types of orbits are generic in a certain sense; in such a case the group and the Lie algebra will be called {\it admissible}. There are examples of nilpotent Lie algebras with flat coadjoint orbits of arbitrary large dimension; actually for every $n$ there are $n$-dimensional admissible groups. In addition, every nilpotent Lie algebra can be embedded in such an admissible Lie algebra \cite[Ex. 4.5.14]{CG}. 

\smallskip
For these admissible groups we prove our main results, Theorem \ref{harnarach} and Corollary \ref{cercif}, that need the notion of the Pedersen-Weyl pseudo-differential calculus. This  is due to \cite{Pe} and is briefly exposed in Subsection \ref{hin}. Let us only say that it is a way to tranform functions or distributions defined on any coadjoint orbit $\O_\xi$ into operators acting on the Hilbert space $\H_\xi$ of the irreducible representation $\xi$ corresponding to this orbit. Defining ${\sf Ped}_\xi$ relies on a special Fourier transform adapted to the coadjoint orbit, involving a predual space $\o_\xi\subset\g$ of $\O_\xi$\,. If flat orbits exist, one can choose the same predual for all of them. Remarkably, if $\G$ is the $(2d+1)$ - dimensional Heisenberg group, Pedersen's quantization (slightly reformulated) is just the Weyl calculus on (generic) $2d$ - dimensional coadjoint obits, that can be identified with the usual phase-space $\R^{2n}$. It also involves naturally a parameter $\lambda\in\R\setminus\{0\}$\,, which can actually be understood as a labelling of the orbits.

\smallskip
Returning to the general case, recall our transformation $\mathscr W$\,, sending scalar functions on $\g^*$ into operator-valued sections on $\wG$\,. Assuming that $\G$ is admissible, $\mathscr W$ can be described in the following way:

\begin{itemize}
\item
Pick a Schwartz function $B:\g^*\to\mathbb C$\,.
\item
Consider all the restrictions $B_\xi\equiv B|_{\O_\xi}$\,, where $\O_\xi$ is the coadjoint orbit corresponding to $\xi\in\wG$\,.
\item
Apply the Pedersen quantization of the orbit to get an operator ${\sf Ped}_\xi(B_\xi)$\,.
\item
Then $[\mathscr W(B)](\xi)={\sf Ped}_\xi(B_\xi)$ for any $\xi\in\wG$\,.
\end{itemize}

Consequently, at least within the class of Schwartz symbols, ${\sf Op}_{\G\times\wG}(\si)={\sf Op}_{\G\times\g^*}(f)$ is equivalent to
$$
\si(x,\xi)={\sf Ped}_\xi\big(f|_{\{x\}\times\O_\xi}\big)\,,\quad\forall\,x\in\G\,,\;\xi\in\wG\,.
$$
This is an extension to admissible groups of a picture that is familiar for the Heisenberg group. The fact that this relation cannot hold for groups without flat coadjoint orbits  is explained at the end of Subsection \ref{horcon}.

\smallskip
We did not use yet the full power of Kirillov's theory. Besides being operator-valued, the functions or sections on the unitary dual $\wG$ are hard to study or to use because $\wG$ itself is, in general, a complicated topological and Borel space, and its Plancherel measure $\wm$ is also complicated and non-explicit. We describe in Subsection \ref{hipolan} how things simplify for admissible groups and how these simplifications are effective at the level of the global group quantization. The emerging concrete picture facilitates a more detailed analytical symbolic calculus, the $\xi\in\wG$ dependence of the symbols being now replaced by the dependence on a parameter belonging to a (large) open subset $\z^*_\bu$ of the dual $\z^*$ of the center of the Lie algebra. This allows us using Euclidean techniques.

\smallskip
A main tool in a pseudo-differential theory is to define, develop and use a symbolic calculus, in the spirit of the traditional one introduced (among others) by L. H\"ormander in the case $\G=\R^n$. Suitable classes $S^m_{\rho,\delta}$ can be found in \cite[Sect.\;5]{FR} for the ${\sf Op}_{\G\times\wG}$ calculus, if $\G$ is a graded group, and many results involving these spaces of symbols are proven. Their definition makes use of Rockland operators (or sub-Laplacians if $\G$ is stratified) and suitable difference operators acting in ``the dual variables". By the combination of partial Fourier transforms making the calculi ${\sf Op}_{\G\times\wG}$ and ${\sf Op}_{\G\times\g^*}$ equivalent, one can transfer constructions and results to the ${\sf Op}_{\G\times\g^*}$ setting, but this is quite implicit. If, in addition, $\G$ is also admissible, one can use the more transparent (and interesting) form of the transformation $\mathscr W$ to get a better understanding. This is done in our Section \ref{hirkovan}, but the conditions defining $S^m_{\rho,\delta}$\,-\, classes on $\G\times\g^*$ still need further investigation.

\smallskip
In Section \ref{hirkan} we do some explicit calculations for a two-step four dimensional admissible graded group. It is shown by direct computations that the relation between ${\sf Op}_{\G\times\wG}$ and ${\sf Op}_{\G\times\g^*}$ holds. The pseudo-differential operators are similar to those appearing for the Heisenberg group; differences appear mainly because now the center of the Heisenberg algebra is $2$-dimensional. Computations of the group Fourier transform and the Pedersen quantizations have simple connections with usual Euclidean Fourier transforms and the parametrized Weyl calculus.

\smallskip
By using automorphisms, one can make the link between different corresponding coadjoint orbits or, equivalently, between different corresponding irreducible representations. This can be raised to links between the Pedersen calculi associated to coadjoint orbits that are connected by automorphisms. In particular, if a family of automorphisms acts transitively on the set of flat coadjoint orbits, computations on just one orbit generate easily formulae for all the others. This appears frequently in the literature, in the form of ``$\lambda$-Weyl calculi" for the Heisenberg group, where actually $\lambda\in\R\!\setminus\!\{0\}$ is a label for the flat coadjoint orbits. We show in an Appendix that this can always be done if the Lie algebra, besides being graded, has a one-dimensional center $\z\equiv\R$ (a condition not satisfied by our example presented in Section \ref{hirkan}). Then the natural dilations associated with the grading, supplemented by an inversion, do act transitively on $\z^*_\bu\equiv\R\!\setminus\!\{0\}$ and we are reduced by simple transformations to the case $\lambda=1$\,. We also present briefly two examples of such Lie algebras, without giving explicitly all the calculations, that can be easily supplied by the reader. They are both graded without being stratified, so, since sub-Laplacians are not available, we indicated in each case a homogeneous positive Rockland operator.

The quantization ${\sf Op}_{\G\times\g^*}$ for invariant operators has been initiated by Melin \cite{Me} and further studied 
by G{\l}owacki in \cite{Glo1,Glo2,Glo3}, and the Weyl type quantizations on nilpotent groups have been analysed by a different approach by Manchon \cite{Ma1,Ma2}. Invariant operators on two-step nilpotent Lie groups have been also studied by \cite{Mi1}. Non-invariant pseudo-differential operators on the Heisenberg group have been analysed in \cite{BKG}, with further extensions to graded nilpotent Lie groups in \cite{FR, FR2}.

\section{Framework}\label{bogart}

We gather here some notions, notations and conventions, including the coadjoint action, its orbits and their predual vector spaces. Flat coadjoint orbits are discussed in some detail.

\subsection{General facts and coadjoint orbits}\label{Perdaf}

For a given (complex, separable) Hilbert space $\H$\,, one denotes by $\mathbb B(\H)$ the $C^*$-algebra of all linear bounded operators in $\H$ and by $\mathbb B^2(\H)$ the bi-sided $^*$-ideal of all Hilbert-Schmidt operators, which is also a Hilbert space with the scalar product $\<A,B\>_{\mathbb B^2(\H)}:={\rm Tr}\!\(AB^*\)$\,. 

\smallskip
Let $\G$ be a connected simply connected nilpotent Lie group with unit $\e$\,, center ${\sf Z}$ and Haar measure $\m$\,. It is second countable, type I and unimodular. On the unitary dual $\wG$ of $\G$ one has the Mackey Borel structure and the Plancherel measure $\wm$\,. Plancherel's Theorem holds \cite{Di}. The elements of $\wG$ are equivalence classes of unitary irreducible strongly continuous Hilbert space representations. {\it In this article we will make a deliberate convenient confusion between classes and elements that represent them.}

\smallskip
Let $\mathfrak g$ be the Lie algebra of $\G$ with center $\mathfrak z={\sf Lie}({\sf Z})$ and $\mathfrak g^*$ its dual. If $X\in\g$ and $\X\in\g^*$ we set $\<X\!\mid\!\X\>:=\X(X)$\,. We also denote by $\exp:\mathfrak g\to\G$ the exponential map, which is a diffeomorphism. Its inverse is denoted by $\log:\G\rightarrow\mathfrak g$\,. Under these diffeomorphisms the Haar measure on $\G$ corresponds to a Haar measure $dX$ on $\g$ (normalized accordingly). It then follows that $L^p(\G)$ would be isomorphic to $L^p(\g)$\,. For each $p\in[1,\infty]$\,, one has an isomorphism 
$$
L^p(\G)\overset{{\rm Exp}}{\longrightarrow} L^p(\mathfrak g)\,,\ {\rm Exp}(u):=u\circ\exp
$$ 
with inverse 
$$
L^p(\mathfrak g)\overset{{\rm Log}}{\longrightarrow} L^p(\G)\,,\ {\rm Log}(\nu):=\nu\circ\log\,.
$$ 
The Schwartz spaces $\S(\G)$ and $\S(\g)$ are defined as in \cite[A.2]{CG}; they are isomorphic Fr\'echet spaces.

\smallskip
For $X,Y\in\mathfrak g$ we set 
\begin{equation*}\label{diosdado}
X\bullet Y:=\log[\exp(X)\exp(Y)]\,.
\end{equation*}
It is a group composition law on $\mathfrak g$\,, given by a polynomial expression in $X,Y$ (the Baker-Campbel-Hausdorff formula). The unit element is $0$ and $X^{\bullet}\equiv -X$ is the inverse of $X$ with respect to $\bullet$\,.

\medskip
Associated to the (unitary strongly continuous) representation $\,\xi:\G\to\mathbb B(\H_\xi)$\,, {\it the space of smooth vectors} 
$$
\H^\infty_\xi:=\big\{\varphi_\xi\in\H_\xi\mid \xi(\cdot)\varphi_\xi\in C^\infty(\G,\H_\xi)\big\}
$$ 
is a Fr\'echet space in a natural way and a dense linear subspace of $\H_\xi$ which is invariant under the unitary operator $\xi(x)$ for every $x\in\G$.  We denote by $\H_\xi^{-\infty}$ the space of all continuous antilinear functionals on $\H_\xi^\infty$ and then we have the natural continuous dense inclusions $\H_\xi^\infty\hookrightarrow\H_\xi\hookrightarrow\H_\xi^{-\infty}$.
 
\smallskip
Now consider the unitary strongly continuous  representation $\xi \otimes {\bar \xi}\colon\G\times\G\to\mathbb B\big[\mathbb B^2(\H_\xi)\big]$ defined by 
$$
(\xi \otimes {\bar \xi})(x_1,x_2)T=\xi(x_1)T\xi(x_2)^{-1},\quad \forall\,x_1,x_2\in\G\,,\,\forall\,T\in\mathbb B^2(\H_\xi)\,.
$$
The corresponding space of smooth vectors is denoted by $\mathbb B(\H_\xi)^\infty$  and is called {\it the space of smooth operators for the representation}~$\xi$. One can prove that $\mathbb B(\H_\xi)^\infty$ is only formed of trace-class operators. Actually \cite{BB7,vDNSZ} we obtain continuous inclusion maps 
\begin{equation*}
\mathbb B(\H_\xi)^\infty\hookrightarrow\mathbb B^1(\H_\xi)\hookrightarrow\mathbb B(\H_\xi)\simeq\big[\mathbb B^1(\H_\xi)\big]^*\hookrightarrow\big[\mathbb B(\H_\xi)^\infty\big]^*=:\mathbb B(\H_\xi)^{-\infty}.
\end{equation*}

The adjoint action \cite{CG,Ki04} is 
$$
{\sf Ad}\colon\G\times\g\to\g\,,\quad {\sf Ad}_x(Y):=\frac{d}{dt}\Big\vert_{t=0}\big[x\exp(tY)x^{-1})\big]
$$
and the coadjoint action of $\G$ is 
$$
{\sf Ad}^*\colon\G\times\g^*\to\g^*, \quad (x,\Y)\mapsto {\sf Ad}^*_x(\Y)\colon=\Y\circ{\sf Ad}_{x^{-1}}. 
$$
Denoting by 
$$
{\rm inn}:\G\times\G\to\G\,,\quad(x,y)\mapsto{\rm inn}_x(y):=xyx^{-1}
$$ 
the action of $\G$ on itself by inner automorphisms, one has 
$$
{\sf Ad}_x=\log\circ\,{\rm inn}_x\circ\exp\,,\quad\forall\,x\in\G\,.
$$
Pick $\,\U\in\g^*$ with its corresponding coadjoint orbit $\,\O(\U):={\sf Ad}_\G^*(\U)\subset\g^*$. {\it The isotropy group at} $\U$ is 
$$
\G_{\U}\!:=\big\{x\in\G\mid{\sf Ad}_x^*(\U)=\U\big\}
$$ 
with the corresponding {\it isotropy Lie algebra} 
\begin{equation*}\label{vezi}
\g_{\U}\!={\sf Lie}(\G_\U)=\{X\in\g\mid \U\circ{\sf ad}_{X}=0\}\supset\z\,.
\end{equation*}
The coadjoint orbit $\O\equiv\O(\U)$ is a closed submanifold and has a polynomial structure comming from its identification with the symmetric space $\G/\G_{\U}$\,. There is a Schwartz space $\S(\O)$ and a Poisson algebra structure on $\g^*$ for which the symplectic leaves are exactly the coadjoint orbits. We refer to \cite{Ki04} for details.
 
\smallskip 
Let $n:=\dim\g$ and fix any sequence of ideals in $\g$, 
$$
\{0\}=\g_0\subset\g_1\subset\cdots\subset\g_n=\g
$$
such that $\dim(\g_j/\g_{j-1})=1$ and $[\g,\g_j]\subset\g_{j-1}$ for $j=1,\dots,n$. Pick any $E_j\in\g_j\!\setminus\!\g_{j-1}$ for $j=1,\dots,n$, so that the set $\mathscr E\!:=\{E_1,\dots,E_n\}$ will be a Jordan-H\"older basis in~$\g$\,. Of course, $\g_j={\rm Span}(E_1,\dots,E_j)$ holds for every $j$\,. The set of \emph{jump indices} of the coadjoint orbit $\O$ with respect to the above Jordan-H\"older basis is  
$$
\epsilon_\O:=\{j\mid \g_j\not\subset\g_{j-1}+\g_{\U}\}=\{j\mid E_j\notin\g_{j-1}+\g_{\U}\}
$$ 
and does not depend on the choice of $\,\U\in\O$\,. The corresponding \emph{predual of the coadjoint orbit}~$\O$ \cite{Pe} is  
$$
\o:={\rm Span}\{E_j\mid j\in \epsilon_\O\}\subset\g
$$
and it turns out that the map $\,\O\ni\Y\mapsto \Y\vert_{\o}\in \o^*$ is a diffeomorphism, explaining the terminology. In addition, one has the direct vector sum decomposition $\g=\g_{\U}\dot +\,\o$\,.

\smallskip
We recall that {\it there is a bijection (even a homeomorphism) between $\wG$ and the family of all coadjoint orbits}; we denote by $\O_\xi$\,, with predual $\o_\xi$\,, the orbit corresponding to the (class of equivalence of the) ireducible representation $\xi:\G\to\mathbb B(\H_\xi)$\,. It is not our intention to review the way this bijection is constructed; see \cite{CG,Ki04} for excellent presentations. But we do recall recall, for further use, a concept that is involved in the construction via the theory of induced representations. The Lie subalgebra $\mathfrak m$ is {\it polarizing} (or {\it maximal subordonate}) to the point $\,\U\in\g^*$ if $\,\U([\mathfrak m,\mathfrak m])=0$ and it is maximal with respect to this property. It is known \cite[Th.\;1.3.3]{CG} that for any point there is at least a polarizing algebra.

\subsection{Flat coadjoint orbits}\label{hiperborean}

A coadjoint orbit $\O$ is called {\it flat} \cite{CG,MW} if $\g_\U=\z$ for some $\U\in\O$\,; then this will also happen for any other element $\Y\in\O$ and thus $\g_\Y$ is an ideal\,. The flatness condition is equivalent with its corresponding irreducible (class of) representation $\xi$ being {\it square integrable modulo the center} and also to the fact that $\,\O=\U+\z^\dag$ for some $\,\U\in\g$ (so it is an affine subspace of $\g^*$)\,; we set 
$$
\z^\dag:=\big\{\Y\in\g^*\mid \Y|_\z=0\big\}
$$ 
for the annihilator of $\z$ in the dual. The orbit only depends on the restriction of $\U$ to $\z$\,. If such orbits $\O$ exist, they are exactly those having maximal dimension. 

\smallskip
Set $\wG_{\bu}$ for the family of (classes of equivalence) of irreducible representations of $\G$ which are square integrable with respect to the center. In many cases $\wG_{\bu}$ is void. But in the opposite cases, when flat coadjoint orbits do exist, the Plancherel measure of $\wG$ is concentrated on $\wG_{\bu}$\,.

\begin{Definition}\label{flatland}
A connected simply connected nilpotent Lie group $\G$ possessing an irreducible uitary representation which is square integrable modulo the center will be called {\rm admissible}. Its Lie algebra $\g$ is also called {\rm admissible}.
\end{Definition}

For the full theory of admissible groups we refer to \cite{MW} and to \cite[Sect. 4.5]{CG}. These groups are not necessarily graded \cite{Bu}. There are criteria for a nilpotent group to have flat coadjoint orbits \cite[Prop. 4.5.9]{CG}. 

\begin{Remark}\label{fletica}
{\rm When a flat orbit $\O$ exists, we can choose an adapted Jordan-H\"older basis. We set $n:=\dim \G$\,, $m:=\dim\z$ and $2d:=\dim\O$\,; thus one has $n=m+2d$\,. Let $\{E_1,\dots,E_m,E_{m+1},\dots,E_n\}$ be a Jordan-H\"older basis of $\g$ such that 
$$
\z={\rm Span}(E_1,\dots,E_m)\,;
$$ 
the jump indices are $\{m+1,\dots,n\}$\,. 
Correspondingly one has 
$$
\g=\mathfrak z\oplus\o\,;
$$ 
the decomposition is orthogonal with respect to the scalar product on $\g$ defined by the basis.  The same decomposition is obtained for any other flat orbit: {\it there is a common predual for all the flat coadjoint orbits}. 
}
\end{Remark}

In the remaining part of this subsection we are going to summarize here some results from \cite{CG,MW}. Let $\G$ be an admissible group and choose a Jordan-H\"older basis $\{E_1,\dots,E_m,E_{m+1},\dots,E_n\}$  of $\g$ as in Remark \ref{fletica}. In terms of annihilators and the dual basis $\{\mathcal E_1,\dots,\mathcal E_m,\mathcal E_{m+1},\dots,\mathcal E_n\}$ in $\g^*$ one has 
$$
{\rm Span}(\mathcal E_1,\dots,\mathcal E_m)=\o^\dag\cong\z^*
$$ 
and 
$$
{\rm Span}(\mathcal E_{m+1},\dots,\mathcal E_n)=\z^\dag\cong\o^*,
$$ 
with 
$$
\g^*=\o^\dag\oplus\z^\dag\cong\z^*\oplus\o^*.
$$ 
Recall the vector space isomorphisms
$$
(\z^\dag)^\perp=\o^\dag\ni\X\to\X|_\z\in\z^*\cong\g^*/\o^*.
$$
Rather often, below, we are going to use the vector space $\mathfrak z^*$ and some of its subsets; in certain situations a more direct interpretation is through the isomorphic version $\o^\dag$.

\smallskip
For any $\,\U\in\g^*$ we define the skew-symmetric bilinear form 
\begin{equation*}\label{striction}
{\rm Bil}_\U:\g\times\g\to\R\,,\quad{\rm Bil}_\U(X,Y):=\<[X,Y]\mid\U\>\,.
\end{equation*} 
If $X\in\z$ or $Y\in\z$ one clearly has ${\rm Bil}_\U(X,Y)=0$\,. Let us denote by ${\rm Bil}_\U^\o$ the restriction of ${\rm Bil}_\U$ to $\o\times\o$\,; it is non-degenerate if and only if the orbit $\O(\U)$ is flat. Its Pfaffian ${\rm Pf}(\U)\equiv{\rm Pf}\big({\rm Bil}_\U^\o\big)\in\R$ is defined by the relation
$$
{\rm Pf}(\U)^2=\det\!\big({\rm Bil}_\U^\o\big)=\det\big[({\sf B}_\U^{ij})_{i,j=1,\dots,2d}\big]\,,
$$
where 
$$
{\sf B}_\U^{i,j}\!:={\rm Bil}_\U(E_{m+i},E_{m+j})=\big\<[E_{m+i},E_{m+j}]\mid\U\big\>\,.
$$ 
The orbit $\,\O(\U)=\U+\mathfrak z^\dag$ of $\U$ is flat if and only if ${\rm Pf}(\U)\ne 0$\,.
It can be checked that ${\rm Pf}(\U)$ only depends on the restriction of $\U$ to $\z$\,, so we get a function 
$$
{\rm Pf}:\z^*\cong\o^\dag\to\R
$$ 
(a $\G$-invariant homogeneous polynomial in the variable $\U\in\o^\dag$)\,.  

\smallskip
Let us set 
$$
\g^*_\bu:=\{\U\in\g^*\mid\O(\U)\ {\rm is\ flat}\}=\{\U\in\g^*\mid{\rm Pf}(\U)\ne 0\}\,.
$$
The family $\wG_{\bu}$ of (classes of equivalence) of irreducible representations of $\G$ which are square integrable with respect to the center is endowed with the restriction of the Fell topology on $\wG$ and with the (full) Plancherel measure. Then Kirillov's homeomorphism $\,\wG\cong\g^*/{\sf Ad}^*$ restricts to 
\begin{equation}\label{ticneala}
\wG_\bu\cong\g^*_\bu/{\sf Ad}^*.
\end{equation} 

Using the center $\mathfrak z$ of the Lie algebra (or the common predual $\o$) we are going to give a more explicit form of \eqref{ticneala}. The subset 
$$
\o^\dag_{\bu}:=\o^\dag\setminus{\rm Pf}^{-1}(0)=\g^*_\bu\cap\o^\dag
$$ 
or, more conveniently, its isomorphic copy
\begin{equation}\label{aciela}
\mathfrak z^*_{\bu}:=\mathfrak z^*\setminus{\rm Pf}^{-1}(0)=\{\Z\in\z^*\mid {\rm Pf}(\Z)\ne 0\}\,,
\end{equation}
with the topological and measure-theoretical structures inherited from the vector space of $\mathfrak z^*$, plays an important role for admissible groups. This is summarised below: 

\begin{Proposition}\label{aiureala}
\begin{enumerate}
\item
The map 
\begin{equation}\label{upsilon}
\Xi:\mathfrak z^*_{\bu}\to\wG_{\bu}\,,\quad\Xi(\Z):=\xi_{\Z+\mathfrak z^\dag}
\end{equation} 
(the equivalence class of irreducible representations associated by Kirillov's theory to the flat coadjoint orbit $\O(\Z)=\Z+\mathfrak z^\dag$) is a homeomorhism.
\item
The Plancherel measure of $\,\wG$ is concentrated on $\wG_{\bu}$\,. Transported back through the bijection $\Xi$\,, it is absolutely continuous with respect to the Lebesgue measure on $\mathfrak z^*_{\bu}\subset\mathfrak z^*$, with density $2^d d!|{\rm Pf}(\Z)|$\,. 
\end{enumerate}
\end{Proposition}

\section{Fourier transformations and the Weyl-Pedersen calculus}\label{Peder}

\subsection{Pedersen quantization of coadjoint orbits}\label{hin}

The arbitrary coadjoint orbit $\O\equiv\O(\U)$ is homeomorphic to the homogeneous space $\G/\G_\U$\,. Since $\G,\G_\U$\,, being both nilpotent, are unimodular, there are ${\sf Ad}^*_{\G}$-invariant measures on $\O$\,; any two of them are connected by multiplication with a strictly positive constant. We refer to \cite[Sect. \!4.2, 4.3]{CG} for information concerning the normalization of the invariant measure in order to fit with Kirillov's Trace Formula.

\smallskip
Let us recall the Fourier transformation associated to a coadjoint orbit (\cite{Pe}). For $\Psi\in\S(\O)$ we set
\begin{equation*}\label{incerp}
\hat{\Psi}:\g\to\mathbb C\,,\quad\hat{\Psi}(X):=\int_{\O} e^{-i\<X\mid \Y\>}\Psi(\Y)\,d\gamma_{\O}(\Y)\,,
\end{equation*}
where $\gamma_{\O}$ is the canonical invariant measure on $\O$ (\cite{CG}). It turns out that $\hat{\Psi}\in C^\infty(\mathfrak g)$ and its restriction to the predual $\,\o$ is a Schwartz function. The map
\begin{equation*}\label{termin}
{\sf F}_{\O}:\S(\O)\to\S(\o)\,,\quad{\sf F}_{\O}(\Psi):=\hat{\Psi}|_{\o}
\end{equation*}
is a linear topological isomorphism called {\it the Fourier transform adapted to the coadjoint orbit $\O$\,}. For some (suitably normalized) Lebesgue measure $\lambda_{\o}$ on $\o$\,, its inverse reads
\begin{equation*}\label{finalizez}
{\sf F}^{-1}_{\O}:\S(\o)\to\S(\O)\,,\quad\big[{\sf F}^{-1}_{\O}(\psi)\big](\Y):=\int_{\o}\!e^{i\<X\mid \Y\>}\psi(X)\,d\lambda_{\o}(X)\,.
\end{equation*}
If the coadjoint orbit is associated to $\xi\in\wG$\,, we use notations as $\O_\xi$\,, $\o_\xi$\,, $\gamma_\xi$\,, $\lambda_\xi$ and ${\sf F}_\xi$\,. Recall that the predual $\o_\xi$ depends on a Jordan-H\"older basis and that the choice of the invariant measure $\gamma_\xi$ fixes $\lambda_\xi$ and ${\sf F}_\xi$\,. Also recall our identification of an irreducible representation with its equivalence class.

\smallskip
If $\,\psi\in\S\big(\o_\xi\big)$ one sets (in weak sense)
\begin{equation*}\label{Dep}
{\sf Dep}_\xi(\psi):=\int_{\o_\xi}\psi(X)\xi(\exp X) d\lambda_{\xi}(X)
\end{equation*}
and then, for $\Psi\in\S\big(\O_\xi\big)$
\begin{equation*}\label{Ped}
{\sf Ped}_\xi(\Psi):={\sf Dep}_\xi\big[{\sf F}_\xi(\Psi)\big]=\int_{\o_\xi}\!\int_{\O_\xi}e^{-i\<X\mid \X\>}\Psi(\X)\xi(\exp X) d\gamma_\xi(\X)d\lambda_{\xi}(X)\,.
\end{equation*}

\noindent
We refer to \cite{Pe} for the properties and the significations of the correspondence $\Psi\mapsto{\sf Ped}_\xi(\Psi)$ and to \cite{BB4,BB5,BB7} for various extensions. In particular, it is known \cite{Pe} that we get a commuting diagram of linear topological isomorphisms
$$
\begin{diagram}
\node{\S\big(\O _\xi\big)}\arrow{e,t}{\sf F_{\xi}} \arrow{s,l}{{\sf Ped}_\xi}\node{\S\big(\o_\xi\big)}\arrow{sw,r}{{\sf Dep}_\xi}\\ 
\node{\mathbb B(\H_\xi)^\infty}
\end{diagram}
$$
If $\Psi\in\S\big(\O_\xi\big)$\,, in terms of the trace ${\rm Tr}_\xi$ on $B^1(\H_\xi)$\, one has 
$$
{\rm Tr}_\xi\big[{\sf Ped}_\xi(\Psi)\big]=\int_{\O_\xi}\!\Psi(\X)d\gamma_\xi(\X)\,.
$$

For the Heisenberg group, by suitable adaptations, one gets in particular the usual Weyl calculus on any of the generic coadjoint orbit. See Subsection \ref{uragan} for anoher explicit example.

\begin{Remark}\label{dequantiz}
{\rm One also has dequantization formulae as 
\begin{equation*}\label{deq1}
{\sf Dep}^{-1}_\xi:\mathbb B(\H_\xi)^\infty\to\S\big(\o_\xi\big)\,,\quad \big[{\sf Dep}^{-1}_\xi(S)\big](X)={\rm Tr}_\xi\big[S\,\xi(\exp X)^*\big]\,,
\end{equation*}
followed by 
\begin{equation*}\label{agricultura}
{\sf Ped}^{-1}_\xi={\sf F}^{-1}_{\xi}\!\circ{\sf Dep}^{-1}_\xi:\mathbb B(\H_\xi)^\infty\to\S\big(\O_\xi\big)\,,
\end{equation*}
\begin{equation}\label{split}
\big[{\sf Ped}^{-1}_\xi(S)\big](\X)=\int_{\o_\xi}\!e^{i\<Y\mid \X\>}{\rm Tr}_\xi\big[S\,\xi(\exp Y)^*\big] d\lambda_\xi(Y)\,.
\end{equation}
}
\end{Remark}

\begin{Remark}\label{extind}
{\rm N. Pedersen showed in \cite[Th 4.1.4.]{Pe} that ${\sf Ped}_\xi$ extends to a topological isomorphism $\S'(\O_\xi)\to\mathbb B(\H_\xi)^{-\infty}$ satifying ${\sf Ped}_\xi(1)=1_\xi$\,, such that for $\Psi_1\in\S(\O_\xi)\,,\,\Psi_2\in\S'(\O_\xi)$ the equality 
\begin{equation*}\label{stramba}
\big\<{\sf Ped}_\xi(\Psi_1),{\sf Ped}_\xi(\Psi_2)\big\>=\Psi_2(\Psi_1)
\end{equation*} 
holds in terms of the duality between $\mathbb B(\H_\xi)^{\infty}$ and $\mathbb B(\H_\xi)^{-\infty}$. If ${\sf Ped}_\xi(\Psi_2)\in\mathbb B(\H_\xi)$ (recall that $\mathbb B(\H_\xi)^{\infty}\subset\mathbb B^1(\H_\xi)$) one even has
\begin{equation*}\label{zdramba}
{\rm Tr}_\xi\big[{\sf Ped}_\xi(\Psi_1){\sf Ped}_\xi(\Psi_2)\big]=\Psi_2(\Psi_1)\,.
\end{equation*} 
}
\end{Remark}

\subsection{Fourier transformations}\label{pieder}

Various Fourier integral formulae will be presented below. For the moment $\G$ is connected simply connected and nilpotent; flat coadjoint orbits are not yet needed.

\medskip
{\bf A.} There is a Fourier transformation, given by the duality $(\g,\g^*)$\,, defined essentially by
\begin{equation*}\label{clata}
\big(\mathscr F_{\!\g,\g^*}h\big)(\X):=\int_{\g}e^{-i\<X\mid \X\>} h(X)\,dX.
\end{equation*}
It is a linear topological isomorphism $\,\mathscr F_{\!\g,\g^*}:\mathcal S(\g)\rightarrow\mathcal S(\g^*)$\,. Using a good normalization of the Lebesgue measure on $\g^*$, it can be seen (after extension) as a unitary map $\,\mathscr F_{\!\g,\g^*}:L^2(\g)\rightarrow L^2(\g^*)$\,.  
 
\medskip
{\bf B.} Composing with the isomorphisms ${\rm Exp}$ and ${\rm Log}$ one gets Fourier transformations
\begin{equation*}\label{fericire}
\mathscr F_{\G,\g^*}:=\mathscr F_{\g,\g^*}\circ{\rm Exp}:\S(\G)\rightarrow \S(\g^*)\,,\quad \mathscr F_{\G,\g^*}^{-1}:={\rm Log}\circ\mathscr F_{\g,\g^*}^{-1}:\S(\g^*)\rightarrow \S(\G)\,,
\end{equation*}
\begin{equation*}\label{qlata}
\big(\mathscr F_{\G,\g^*} u\big)(\X)=\int_{\g}e^{-i\<X\mid \X\>}u(\exp X)dX=\int_\G e^{-i\<\log x\mid \X\>} u(x)d\m(x)\,,
\end{equation*}
\begin{equation}\label{qluta}
\big(\mathscr F_{\G,\g^*}^{-1}w\big)(x)=\int_{\g^*}\!e^{i\<\log x\mid \X\>}w(\X)d\X.
\end{equation}
These maps can also be regarded as unitary isomorphisms of the corresponding $L^2$-spaces.

\medskip
{\bf C.} One also has the unitary group Fourier transform 
$$
\mathscr F_{\G,\wG}:L^2(\G)\rightarrow \mathscr B^2(\wG):=\int^\oplus_{\wG}\,\mathbb B^2(\H_\xi)\,d\wm(\xi)\,,
$$
defined on $L^1(\G)\cap L^2(\G)$ as
\begin{equation*}\label{ferikire}
\big(\mathscr F_{\!\G,\wG}\,u\big)(\xi):=\int_\G u(x)\xi(x)^*d\m(x)\,,
\end{equation*}
with inverse (on sufficiently regular elements $b$)
\begin{equation}\label{cerifire}
\big(\mathscr F^{-1}_{\!\G,\wG}\,b\big)(x):=\int_{\wG}{\rm Tr}_\xi[b(\xi)\xi(x)]d\wm(\xi)\,.
\end{equation}
It also becomes an isomorphism of Schwartz-type spaces, if we simply define $\mathscr S(\wG)$ to be the image of $\S(\G)$ through $\mathscr F_{\!\G,\wG}$ with the transported topological structure;
the space $\mathscr S(\wG)$ is difficult to describe explicitly (see \cite{Ge}).

\subsection{The transformation $\mathscr W$}\label{horcon}

We are now interested in the mapping
\begin{equation*}\label{firsta}
\mathscr W:=\mathscr F_{\!\G,\wG}\circ\mathscr F_{\!\G,\g^*}^{-1}:\S(\g^*)\to\mathscr S(\wG)
\end{equation*}
and its inverse. If $\g$ is Abelian, identifying $\wG$ with $\g^*$, it can be seen as the identity mapping. For reasons that will be exposed below (see Remark \ref{ludwig} for instance), we will restrict ourselves to admissible groups. As we will see, {\it in this case $\mathscr W$ basically consists in restricting the element of $\S(\g^*)$ to all the coadjoint orbits and then applying the corresponding Pedersen quantizations to these restrictions}.

\begin{Theorem}\label{harnarach}
Let $\G$ be an admissible group.
\begin{enumerate}
\item[(i)]
For $\,B\in\S(\g^*)$ and $\xi\in\wG$ set 
$$
B_\xi:=B|_{\O_\xi}\ \ \ {\it and}\ \ \ b(\xi):={\sf Ped}_\xi\big(B_\xi\big)\,.
$$
Then $\,b(\xi)\in\mathbb B(\H_\xi)^\infty$ and one has 
\begin{equation}\label{walioasa}
b=\mathscr W(B)\in\mathscr S(\wG)\,.
\end{equation}
\item[(ii)]
Conversely, let 
$$
b\equiv\big\{b(\xi)\mid \xi\in\wG\big\}\in \mathscr S(\wG)\,.
$$ 
For every $\xi\in\wG$ and every $\X\in\O_\xi$ one has
\begin{equation}\label{mormulla}
\big[\mathscr W^{-1}(b)\big](\X)=\int_{\o_\xi}e^{i\<Y\mid\X\>}\,{\rm Tr}_\xi\big[b(\xi)\xi(\exp Y)^*\big]d\lambda_\xi(Y)\,.
\end{equation}
\end{enumerate}
\end{Theorem}

We are going to need two lemmas. The first one gives a first (rather weak) control on the map $\xi\mapsto{\sf Ped}_\xi\big(B_\xi\big)$\,. The direct integral Banach space $\mathscr B^1(\wG)$ is defined similarly to $\mathscr B^2(\wG)$\,, but with respect to the norm
\begin{equation*}\label{be11}
\p\!\phi\p_{\mathscr B^1(\wG)}\,:=\int_{\wG}\p\!\phi(\xi)\!\p_{\mathbb B^1(\H_\xi)}\!d\wm(\xi)\,.
\end{equation*} 

\begin{Lemma}\label{primoa}
For any $B\in\mathcal D(\g^*)\equiv C^\infty_{\rm c}(\g^*)$ one has $b(\cdot)\in\mathscr B^1(\wG)\cap\mathscr B^2(\wG)$\,.
\end{Lemma}

\begin{proof}
Set $K:={\rm supp}(B)$ (a compact subset of $\g^*$)\,. Recall that $\wG$ is homeomorphic to the quotient of $\g^*$ by the coadjoint action: one has $\,\g^*\overset{q}{\longrightarrow}\g^*/_{{\sf Ad}^*}\cong\wG$\,.
Then for $\O_\xi$ not belonging to the (quasi-)compact subset $K':=q(K)$ of $\g^*/_{{\sf Ad}^*}$\,, meaning that $\O_\xi\cap K=\emptyset$\,, one has $B|_{\O_\xi}=0$\,. Thus $b(\xi)=0$ if $\xi$ belongs to the complement of the homeomorphic compact image of $K'$ in $\wG$\,. 

\smallskip
The Plancherel measure is bounded on compact subsets \cite[18.8.4]{Di}, so one gets $b\in\mathscr B^1(\wG)\cap\mathscr B^2(\wG)$ if $\,\xi\mapsto{\sf Ped}_\xi\big(B_\xi\big)$ is essentially locally bounded. Since the Plancherel measure is concentrated on the set of generic flat orbits, only the corresponding irreducible representations $\xi$ are important. Then local boundedness follows easily from 
\cite[Th.\;4.4]{BB5} and the fact that $B\in\mathcal D(\mathfrak g^*)$\,.
\end{proof}

\begin{Lemma}\label{adooa}
For any $C\in\S(\g^*)$ one has
\begin{equation}\label{streeka}
\int_{\wG}\Big[\,\int_{\O_\xi}\!C|_{\O_\xi}(\X)d\gamma_\xi(\X)\Big]d\wm(\xi)=\int_{\g^*}\!C(\X)d\X.
\end{equation}
\end{Lemma}

\begin{proof}
See \cite[Pag. 153-154]{CG} or \cite[Page 100]{Ki04}.
\end{proof}

One can now prove Theorem \ref{harnarach}.

\begin{proof}
(i) Taking into account the way $\mathscr W$ is defined, the identity $\,b=\mathscr W(B)$ is equivalent to 
$$
\mathscr F^{-1}_{\!\G,\wG}(b)=\mathscr F^{-1}_{\G,\g^*}(B)\,.
$$ 
It is enough to assume $B\in\mathcal D(\g^*)$\,; then clearly $B_\xi\in\mathcal D(\O_\xi)\subset\S(\O_\xi)$ and thus 
$$
b(\xi)\in\mathbb B(\H_\xi)^\infty\subset\mathbb B^1(\H_\xi)\subset\mathbb B^2(\H_\xi)\,.
$$
In the computation below we will need to apply formula \eqref{cerifire} as it is (pointwise). Recall (cf. \cite{Fu}) that $\mathscr F_{\!\G,\wG}$ restricts to an isomorphism 
$$
L^1(\G)\cap A(\G)\to \mathscr B^1(\wG)\cap\mathscr B^2(\wG)\,,
$$ 
where $A(\G)$ is Eymard's Fourier algebra. In addition, if $\phi\in\mathscr B^1(\wG)\cap\mathscr B^2(\wG)$\,, the inversion formula \eqref{cerifire} holds pointwisely. But it is shown in Lemma \ref{primoa} that $b(\cdot)\in\mathscr B^1(\wG)\cap\mathscr B^2(\wG)$\,, so we have pointwisely
\begin{equation}\label{ztupit}
\Big[\mathscr F^{-1}_{\!\G,\wG}(b)\Big](x)=\int_{\wG} {\rm Tr}_\xi\big[b(\xi)\xi(x)\big]d\wm(\xi)\,.
\end{equation}
We work with $x=\exp X$\,; by Remark \ref{extind}, there is a unique distribution $\Phi_X^{(\xi)}\in\S'(\O_\xi)$ such that 
$$
\xi(\exp X)={\sf Ped}_\xi\big(\Phi_X^{(\xi)}\big)\,,
$$ 
with
\begin{equation*}\label{famen}
{\rm Tr}_\xi\big[b(\xi)\xi(\exp X)\big]={\rm Tr}_\xi\big[{\sf Ped}_\xi(B_\xi){\sf Ped}_\xi\big(\Phi_X^{(\xi)}\big)\big]=\Phi_X^{(\xi)}(B_\xi)\,.
\end{equation*}
Computing the Pedersen symbol $\Phi_X^{(\xi)}$ of $\xi(\exp X)$ in general seems to be difficult. But if $\G$ is admissible, since $\wG\setminus\wG_\bu$ is $\wm$-negligible, in \eqref{ztupit} we can concentrate on flat orbits and use a result from \cite{Pe}. 

\smallskip
Assuming that $\O_\xi=\U+\z^\dag$ is flat, let us decompose 
$$
X=X_\z+X_{\o}\in\z\oplus\o=\g
$$ 
(see Remark \ref{fletica}). Note that $\exp X=\exp X_\z\exp X_{\o}$\,; higher order terms in the BCH formula are trivial, since $X_\z$ is central. {\it The central character} $\chi_\xi:{\sf Z}\to\mathbb T$ of the irreducible representation $\xi$ is defined by 
$$
\xi(z)=\chi_\xi(z)1_\xi\,,\ \ \ \forall\,z\in{\sf Z}
$$ 
and is given by $\chi_\xi(z)=e^{i\<\log z\mid\U\>}$ (independent on the choice of the point $\U\in\O_\xi$)\,; thus we have 
\begin{equation*}\label{zduvoz}
\xi(\exp X)=\xi\big(\exp X_\z\big)\xi\big(\exp X_{\o}\big)=e^{i\<X_\z\mid\U\>}\xi\big(\exp X_{\o}\big)\,.
\end{equation*}
Then (some steps will be explained below)
\begin{equation}\label{aigrija}
\begin{aligned}
{\rm Tr}_\xi\big[{\sf Ped}_\xi(B_\xi)\xi(\exp X)\big]&=e^{i\<X_\z\mid\U\>}{\rm Tr}_\xi\big[{\sf Dep}_\xi\big({\sf F}_\xi(B_\xi)\big)\xi\big(\exp X_{\o}\big)\big]\\
&=e^{i\<X_\z\mid\U\>}\big({\sf F}_\xi(B_\xi)\big)(-X_{\o})\\
&=e^{i\<X_\z\mid\U\>}\int_{\O_\xi}\!e^{i\<X_{\o}\mid\X\>}B_\xi(\X)d\gamma_\xi(\X)\\
&=\int_{\O_\xi}\!e^{i\<X\mid\X\>}B_\xi(\X)d\gamma_\xi(\X)\,.
\end{aligned}
\end{equation}
The second equality is equivalent to Lemma 4.1.2 from \cite{Pe}, relying on the deep result \cite[Th.\,2.1.1]{Pe}; note that ${\sf Dep}_\xi$ corresponds to the notation $T$ of Pedersen, and the extra constant in  \cite[Lemma 4.1.2]{Pe} is due to different conventions. For the last equality recall that $\O_\xi=\U+\z^\perp$, which allows one to write for each $\X\in\O_\xi$
$$
\<X_\z\!\mid\U\>+\big\<X_{\o}\!\mid\!\X\big\>=\<X_\z\!\mid\!\X\>+\big\<X_{\o}\!\mid\!\X\big\>=\<X\!\mid\!\X\>\,.
$$
Replacing this above and also recalling \eqref{streeka} and \eqref{qluta}, one gets
\begin{equation*}\label{prost}
\begin{aligned}
\big[\mathscr F^{-1}_{\!\G,\wG}(b)\big](\exp X)&
=\int_{\wG_\bu}\Big[\,\int_{\O_\xi}\!e^{i\<X\mid \X\>}B_\xi(\X)d\gamma_\xi(\X)\Big]d\wm(\xi)\\
&=\int_{\g^*}\!e^{i\<X\mid \X\>}B(\X)d\X=\big[\mathscr F^{-1}_{\G,\g^*}(B)\big](\exp X)\,.
\end{aligned}
\end{equation*}
Now, once the identity $b=\mathscr W(B)$ is proven, the fact that $\,b\in\mathscr S(\wG)$ is clear from $\mathscr S(\wG):=\mathscr F_{\!\G,\wG}[\S(\G)]$ and the definition of $\mathscr W$.

\medskip
(ii) To compute $\mathscr W^{-1}$, seen as the inverse of $\mathscr W$ already made explicit at point 1, we have to use the dequantisation formulae of Remark \ref{dequantiz}: First recall that $b\in\mathscr S(\wG)$ means by definition that 
$$
b=\mathscr F_{\!\G,\wG}(c)\,,\ \ {\rm with}\ \ c\in\S(\G)\,.
$$ 
In \cite[Th. 3.4]{Ho} it is shown that for every $\xi$ the map $c\mapsto\big[\mathscr F_{\!\G,\wG}(c)\big](\xi)$ sends $\S(\G)$ (surjectively but not injectively) to $\mathbb B(\H_\xi)^\infty$\,.  Therefore $b(\xi)$ belongs to $\mathbb B(\H_\xi)^\infty$ and  we can construct 
$$
B_{\xi}:={\sf Ped}^{-1}_\xi[b(\xi)]\in\S\big(\O_\xi\big)
$$ 
given by \eqref{split}. For $\X\in\g^*$ we put $B(\X):=B_{\xi}(\X)$\,, selecting $\xi$ such that $\X\in\O_\xi$\,. Then \eqref{split} leads finally to the formula \eqref{mormulla} for $B=\mathscr W^{-1}(b)$\,.
\end{proof}

Let us briefly indicate how a {\it weaker form} of the point (i) of Theorem \ref{harnarach} follows from a result in \cite{Pe}.

\begin{proof}
Using our notations, Theorem 4.2.1 in \cite{Pe} states that for every $u\in\S(\G)$ and every $\xi\in\wG$ which is square integrable modulo the center, one has
\begin{equation*}\label{state}
\big[\mathscr F_{\!\G,\wG}(u)\big](\xi)={\sf Ped}_\xi\big[\mathscr F_{\g,\g^*}\!\big(u\circ\exp\big)|_{\O_\xi}\big]\,.
\end{equation*}
Setting $\,u:=\big[\mathscr F^{-1}_{\g,\g^*}(B)\big]\!\circ\log=\mathscr F^{-1}_{\G,\g^*}(B)$ we get for such $\xi$
\begin{equation*}\label{minuny}
[\mathscr W(B)](\xi):=\mathscr F_{\G,\wG}\big[\mathscr F^{-1}_{\G,\g^*}(B)\big](\xi)={\sf Ped}_\xi\,(B|_{\O_\xi}\,)=:b(\xi)\,.
\end{equation*}
Thus we checked that {\it $\mathscr W(B)$ and $b$ defined in Theorem \ref{harnarach} coincide on the set of square integrable modulo the center irreducible representations}. However, it seems rather hard to go further. Neither regularity properties of the implicitly defined space $\mathscr S(\wG)$\,, nor the smoothness of $b$ are obvious a priori.
\end{proof}

\begin{Corollary}\label{situatie}
Let $\xi$ be an irreducible representation of the admissible group $\G$\,, that is square integrable modulo the center, and let 
$$
\widetilde \xi(v):=\int_\G v(x)\xi(x)d\m(x)\,,\quad v\in\S(\G),
$$
its integrated form acting on the Schwartz space. One has
\begin{equation*}\label{dwig}
\ker\big(\widetilde\xi\,\big)=\Big\{v\in\S(\G)\,\big\vert\,\big[\mathscr F_{\G,\g^*}(v)\big]|_{\O_\xi}=0\Big\}\,.
\end{equation*}
\end{Corollary}

\begin{proof}
Let $\,v:=\mathscr F^{-1}_{\G,\g^*}(B)\in\S(\G)$ with $B=\mathscr F_{\G,\g^*}\!(v)\in\S(\g^*)$ and set as above $b(\eta):={\sf Ped}_\eta\big(B|_{\O_\eta}\big)$ for every $\eta\in\wG$\,. With these notations, and using our result \eqref{walioasa}, one has 
\begin{equation*}\label{apropii}
\widetilde\xi(v)=\Big(\mathscr F_{\G,\wG}\big[\mathscr F^{-1}_{\G,\g^*}(B)\big]\Big)(\xi)=b(\xi)\,.
\end{equation*}
Since the map ${\sf Ped}_\xi$ is an isomorphism, one has $\widetilde\xi(v)=0$ if and only if $b(\xi)=0$ and if and only if $B|_{\O_\xi}=\mathscr F^{\Phi}_{\G,\g^*}\!(v)|_{\O_\xi}=0\,.$
\end{proof}

\begin{Remark}\label{ludwig}
{\rm Working with a general connected simply connected nilpotent Lie group $\G$\,, Ludwig \cite{Lu} defines a (two-sided self-adjoint) ideal $\mathcal J$ of $L^1(\G)$ to be {\it good}\, if ${\rm Exp}(\mathcal J)$ is an ideal in $L^1(\g)$ with the obvious convolution multiplication. Let 
\begin{equation*}\label{formtegrata}
\widetilde\xi:\S(\G)\to\mathbb B(\H_\xi)^\infty,\quad \widetilde\xi(v):=\big[\mathscr F_{\!\G,\wG}(v)\big](\xi)=\int_\G v(x)\xi(x)^* d\m(x)
\end{equation*}
the integrated form of $\xi\in\wG$\,. Then, for an element $\xi$ of $\wG$\,, he shows that $\ker \widetilde\xi\,$ is good if and only if $\O_\xi$ is an affine subspace and if and only if 
$$
\ker\big(\widetilde\xi\,\big)=\big\{v\in L^1(\G)\mid \mathscr F_{\G,\g^*}(v)|_{\O_\xi}=0\big\}\,.
$$ 
It is easy to see that this forbids Theorem \ref{harnarach} to hold for non-admissible groups. This has its roots in the form of the Pedersen symbol $\Phi_X^{(\xi)}$ of $\xi(\exp X)$ for general $\xi$ and it is probably related to the necessity of introducing a modified Fourier transformation instead of $\mathscr F_{\G,\g^*}$.}
\end{Remark}

\section{Various quantizations and their mutual connections}\label{fifirin}

In this section we discuss different quantizations on $\G$ and relations among them.

\subsection{A list of quantizations}\label{fifitrin}

One has various quantizations on the "phase spaces" $\mathfrak g\times\mathfrak g^*\ni(X,\X)$\,, $\G\times\g^*\ni(x,\X)$\,, $\G\times\wG\ni(x,\xi)$\,: 

\begin{equation}\label{garaci}
\begin{aligned}
&\Op_{\g\times\g^*}:L^2(\mathfrak g\times\mathfrak g^*)\to\mathbb B^2\big[L^2(\g)\big]\,,\\
\big[&\Op_{\g\times\g^*}({\sf f})\nu\big](X)=\int_\g\int_{\g^*}\!e^{i\<(-Y)\bu X\mid \X\>}{\sf f}\!\left(X,\X\right)\nu(Y)\,dY\,d\X
\end{aligned}
\end{equation}
and
\begin{equation}\label{daraci}
\begin{aligned}
&{\sf Op}_{\g\times\g^*}:L^2(\mathfrak g\times\mathfrak g^*)\to\mathbb B^2\big[L^2(\G)\big]\,,\quad {\sf Op}_{\g\times\g^*}({\sf f})={\rm Exp}\circ\Op_{\g\times\g^*}({\sf f})\circ{\rm Log}\,,\\
\big[&{\sf Op}_{\g\times\g^*}({\sf f})u\big](x)=\int_\G\int_{\g^*}\!e^{i\<\log(y^{-1}x)\mid \X\>}{\sf f}\big(\log x,\X\big)u(y)\,d\m(y)\,d\X
\end{aligned}
\end{equation}
and the one we prefer
\begin{equation}\label{tulcisor}
\begin{aligned}
&{\sf Op}_{\G\times\g^*}:={\sf Op}_{\g\times\g^*}\circ({\rm Exp}\otimes{\rm id}):L^2(\G\times\mathfrak g^*)\to\mathbb B^2\big[L^2(\G)\big]\,,\\
\big[&{\sf Op}_{\G\times\g^*}\!(f)u\big](x)=\int_\G\int_{\g^*}\!e^{i\<\log(y^{-1}x)\mid \X\>}f\big(x,\X\big)u(y)\,d\m(y)\,d\X\,.
\end{aligned}
\end{equation}

\begin{Remark}\label{account}
{\rm Taking into account Schwartz's Kernel Theorem and the way various Schwartz spaces were defined, one gets topological linear isomorphisms
\begin{equation*}\label{duaka}
{\sf Op}_{\g\times\g^*}:\S(\mathfrak g\times\mathfrak g^*)\overset{\sim}{\longrightarrow}\mathbb B^2\big[\S'(\G),\S(\G)\big]\,,\quad
{\sf Op}_{\g\times\g^*}:\S'(\mathfrak g\times\mathfrak g^*)\overset{\sim}{\longrightarrow}\mathbb B^2\big[\S(\G),\S'(\G)\big]\,,
\end{equation*}
\begin{equation*}\label{treiaka}
{\sf Op}_{\G\times\g^*}:\S(\G\times\mathfrak g^*)\overset{\sim}{\longrightarrow}\mathbb B^2\big[\S'(\G),\S(\G)\big]\,,\quad
{\sf Op}_{\G\times\g^*}:\S'(\G\times\mathfrak g^*)\overset{\sim}{\longrightarrow}\mathbb B^2\big[\S(\G),\S'(\G)\big]\,.
\end{equation*}
}
\end{Remark}

Finally, recall that we were interested in the global group quantization  \cite{FR,MR} (irreducible representations are still identified to corresponding equivalence classes)
\begin{equation}\label{ploicica}
\begin{aligned}
{\sf Op}_{\G\times\wG}&:\mathscr B^2(\G\times\wG):=L^2(\G)\otimes\mathscr B^2(\wG)\to\mathbb B^2\big[L^2(\G)\big]\,,\\
\big[{\sf Op}_{\G\times\wG}(\si)u\big](x)&=\int_\G\int_{\wG}{\rm Tr}_\xi\big[\xi(y^{-1}x)\si(x,\xi)\big]u(y)\,d\m(y)d\wm(\xi)\,.
\end{aligned}
\end{equation}

\subsection{Connections among quantizations}\label{han}

It has been shown in \cite{MR} that ${\sf Op}_{\G\times\wG}$ and ${\sf Op}_{\G\times\g^*}$ are also equivalent, for general connected simply connected nilpotent Lie groups. Actually one has the following commutative diagram of isomorphisms:
$$
\begin{diagram}
\node{L^2(\G)\otimes L^2(\G)}\arrow{e,t}{{\sf id}\otimes{\mathscr F}_{\G,\wG}} \arrow{s,l}{{\sf id}\otimes\mathscr F_{\G,\g^*}}\node{L^2(\G)\otimes\mathscr B^2(\wG)}\arrow{s,r}{{\sf Op}_{\G\times\wG}}\\ 
\node{L^2(\G)\otimes L^2(\mathfrak g^*)} \arrow{e,t}{{\sf Op}_{\G\times\g^*}} \arrow{ne,l}{{\rm id}\otimes\mathscr W}\node{\mathbb B^2[L^2(\G)]}
\end{diagram}
$$
The reason is that one can write
\begin{equation*}\label{rait}
{\sf Op}_{\G\times\wG}={\sf Int}\circ{\rm CV}\circ\big({\sf id}\otimes\mathscr F_{\!\G,\wG}^{-1}\big)\quad{\rm and}\quad{\sf Op}_{\G\times\g^*}={\sf Int}\circ{\rm CV}\circ\big({\sf id}\otimes\mathscr F_{\G,\g^*}^{-1}\big)\,,
\end{equation*}
where, besides the Fourier transformations already introduced, ${\sf Int}(K)$ is the integral operator of kernel $K$ and ${\rm CV}$ means composition with the change of variables
\begin{equation*}\label{change}
{\rm cv}:\G\times\G\to\G\times\G\,,\quad{\rm cv}(x,y):=\big(x,y^{-1}x\big)\,.
\end{equation*}
Using the isomorphism $\mathscr W$, for $A\in L^2(\G)\otimes L^2(\mathfrak g^*)$ one has
\begin{equation}\label{sfaura}
{\sf Op}_{\G\times\wG}\big[\big({\rm id}\otimes\mathscr W\big)f\big]={\sf Op}_{\G\times\g^*}(f)\,.
\end{equation}
This is quite easy to handle for admissible groups, since $\mathscr W$ has the simple form given by Theorem \ref{harnarach}.

\begin{Corollary}\label{cercif}
Assume that $\G$ is an admissible group. For $f\in\S(\G\times\g^*)$ and $(x,\xi)\in\G\times\wG$\,, let us set $f_{(x,\xi)}\in\S(\O _\xi)$ through 
$$
f_{(x,\xi)}(\X):=f(x,\X)\quad {\rm for\ every}\ \ \X\in\O_\xi\,,
$$ 
i.e. $f_{(x,\xi)}$ is the restriction of $f$ to the subset $\{x\}\times\O_\xi$\,, seen as a function $\O_\xi\to\mathbb C$\,. 
The expression
\begin{equation*}\label{real}
\si(x,\xi):={\sf Ped}_\xi\big(f_{(x,\xi)}\big)=\int_{\o_\xi}\!\big[{\sf F}_\xi\big(f_{(x,\xi)}\big)\big](Y)\,\xi(\exp Y)\,d\lambda_\xi(Y)\in\mathbb B(\H_\xi)^\infty
\end{equation*}
denotes the Pedersen quantization of this symbol, associated to the coadjoint orbit $\O_\xi$\,. Then one has 
$$
\big({\rm id}\otimes\mathscr W\big)(f)=\si\quad {\rm and}\quad{\sf Op}_{\G\times\wG}(\si)={\sf Op}_{\G\times\g^*}(f)\,.
$$
\end{Corollary}

\begin{proof}
This follows from Theorem \ref{harnarach} and from \eqref{sfaura}.
\end{proof}

\begin{Remark}\label{industrie}
{\rm We described the correspondence $\S(\G\times\g^*)\ni f\mapsto \si\in\mathscr S(\G\times\wG)$\,. For the reverse one, we have to use the dequantization formulae of Remark \ref{dequantiz}: Suppose we are given 
$$
\si\equiv\big\{\si(x,\xi)\in\mathbb B(\H_\xi)^\infty\mid (x,\xi)\in\G\times\wG\big\}\in \mathscr S(\G\times\wG)\,.
$$ 
For each $(x,\xi)$ we construct 
$$
f_{(x,\xi)}:={\sf Ped}^{-1}_\xi[\si(x,\xi)]\in\S\big(\O_\xi\big)
$$ 
given by \eqref{split} and then, for $(x,\X)\in\G\times\g^*$ with $\X\in\O_\xi$, we put 
$$
f(x,\X):=f_{(x,\xi)}(\X)\,.
$$ 
This leads finally to the formula
\begin{equation*}\label{ormula}
f(x,\X)=\int_{\o_\xi}e^{i\<Y\mid\X\>}\,{\rm Tr}_\xi\big[\si(x,\xi)\xi(\exp Y)^*\big]d\lambda_\xi(Y)\,,\quad\forall\,(x,\X)\in\G\times\O_\xi\,.
\end{equation*}
}
\end{Remark}

\begin{Remark}\label{zimilar}
{\rm Obviously one can introduce composition laws
$$
\S(\O_\xi)\!\times\!\S(\O_\xi)\overset{\sharp_\xi}{\longrightarrow}\S(\O_\xi)\,,\quad{\sf Ped}_\xi(\Psi_1\sharp_\xi\Psi_2):={\sf Ped}_\xi(\Psi_1){\sf Ped}_\xi(\Psi_2)\,,
$$
$$
\mathscr S(\G\times\wG)\!\times\!\mathscr S(\G\times\wG)\overset{\#_{\G\times\wG}}{\longrightarrow}\mathscr S(\G\times\wG)\,,\quad{\sf Op}_{\G\times\wG}\big(\si_1\,\#_{\G\times\wG}\,\si_2\big):={\sf Op}_{\G\times\wG}(\si_1){\sf Op}_{\G\times\wG}(\si_2)\,,
$$
$$
\S(\G\times\g^*)\!\times\!\S(\G\times\g^*)\overset{\#_{\G\times\g^*}}{\longrightarrow}\S(\G\times\g^*)\,,\quad{\sf Op}_{\G\times\mathfrak g^*}\big(f_1\,\#_{\G\times\g^*} f_2\big):={\sf Op}_{\G\times\g^*}(f_1){\sf Op}_{\G\times\g^*}(f_2)\,.
$$
Denoting ${\rm id}\otimes\mathscr W$ by $\mathfrak W$ one has
\begin{equation*}\label{gonegt}
\mathfrak W\big(f_1\,\#_{\G\times\g^*} f_2\big)=\mathfrak W(f_1)\,\#_{\G\times\wG}\,\mathfrak W(f_2)\,,\quad\forall\,f_1,f_2\in\S(\G\times\g^*)\,.
\end{equation*}
One can write explicit (but rather complicated) formulae for these composition rules. 
In the case of ${\sf Op}_{\G\times\wG}$, see \cite{FR2} for a detailed discussion.
Similarly for involutions. These $^*$-algebras and their extensions will be studied separately.}
\end{Remark}

\begin{Remark}\label{aman}
{\rm We make a formal statement about how convolution operators (in $\S(\G)$\,, $L^2(\G)$ or other function spaces on $\G$) fit in the setting above; this can be made rigorous under suitable assumptions. Let us set
\begin{equation*}\label{conv}
\big[{\sf Conv}_R(w)u\big](x):=(u\ast w)(x)=\int_\G u(y)w(y^{-1}x)d\m(y)\,.
\end{equation*}
It is easy to show that, for suitable $B:\g^*\to\mathbb C$\,, one has
\begin{equation}\label{anshon}
{\sf Conv}_R\big[\mathscr F^{-1}_{\G,\g^*}(B)\big]={\sf Conv}_R\big[\mathscr F^{-1}_{\!\g,\g^*}(B)\circ\log\big]={\sf Op}_{\G\times\g^*}(1\otimes B)\,.
\end{equation}
Such left-invariant (or similar right-invariant) operators on various types of nilpotent Lie groups $\G$ were studied in detail in \cite{Glo1,Glo2,Glo3,Me,Mi,Mi1} and other references. 

\smallskip
On the other hand, in \cite[Sect. 7.3]{MR} we proved that for convenient sections $b$ over $\wG$ one gets
\begin{equation}\label{ganshon}
{\sf Conv}_R\big[\mathscr F^{-1}_{\!\G,\wG}(b)\big]={\sf Op}_{\G\times\wG}(1\otimes b)\,.
\end{equation}
Of course, this is compatible with Corollary \ref{cercif}. Writing \eqref{anshon} and \eqref{ganshon} as
\begin{equation*}\label{manshon}
{\sf Conv}_R(w)={\sf Op}_{\G\times\g^*}\big(1\otimes \big[\mathscr F_{\G,\g^*}(w)\big]\big)={\sf Op}_{\G\times\wG}\big(1\otimes \big[\mathscr F_{\G,\wG}(w)\big]\big)\,,
\end{equation*}
one could say that, in particular, ${\sf Op}_{\G\times\g^*}$ and ${\sf Op}_{\G\times\wG}$ are two different but related ways to study invariant operators through symbolic calculi. In the first case the symbols are scalar and defined on the dual of the Lie algebra, in the second case they are defined on the unitary dual of the group and are operator-valued. The same atitude towards non-invariant operators leads to the full quantizations \eqref{tulcisor} and \eqref{ploicica} with ``variable coefficients" pseudo-differential operators.
}
\end{Remark}

\subsection{The concrete Fourier transform and concrete quantizations}\label{hipolan}

The effect of the constructions and results described in Subsection \ref{hiperborean} is that in the admissible case, for many purposes, one can replace the rather abstract and inaccessible measure space $\big(\wG,\wm\big)$ by $\big(\mathfrak z^*_\bu,\mu\big)$\,, where $\z^*_\bu\cong\o^\dag_\bu$ is a subset of a finite-dimensional real vector space and $\mu$ a measure defined by an explicit density.

\smallskip
Taking advantage of Proposition \ref{aiureala}, if $F$ is a function on $\wG_{\bu}$\,, we turn it into a function on $\mathfrak z^*_\bu$ by $\tilde\Xi(F):=F\circ\Xi$\,. Similarly, $\tilde\Xi^{-1}(G):=G\circ\Xi^{-1}$ is a function on $\wG$ if $G$ is a function on $\mathfrak z^*_\bu$\,. The same works for sections in fiber bundles over the two spaces. Topological vector spaces of sections over $\wG_{\bu}$ (as those defined over $\wG$ but insensible to removing the negligible subset $\wG\!\setminus\!\wG_\bu$) are transferred to similar topological vector spaces of sections over $\mathfrak z^*_\bu$\,. One has, for instance, the family of Banach spaces $\mathscr B^p\big(\mathfrak z^*_\bu\big)$ indexed by $p\in[1,\infty)$\,, in particular, the Hilbert space
\begin{equation*}\label{caterus}
\mathscr B^2\big(\mathfrak z^*_\bu\big)=\int^\oplus_{\mathfrak z^*_\bu}\!\mathbb B^2\big(\H_{\Xi(\Z)}\big)\,2^d d!|{\rm Pf}(\Z)|d\Z\,.
\end{equation*}

\smallskip
Supposing that flat coadjoint orbits exist, with generic dimension $2d$\,, one has
\begin{equation*}\label{supliment}
\mathscr F_{\G,\mathfrak z^*_{\bullet}}:=\tilde\Xi\circ\mathscr F_{\G,\wG}:L^2(\G)\rightarrow \mathscr B^2(\mathfrak z^*_{\bullet})
\end{equation*}
defined as 
\begin{equation*}\label{ferigire}
\big[\mathscr F_{\!\G,\mathfrak z^*_{\bullet}}(u)\big](\Z):=\int_\G u(x)\xi_{\Z+\mathfrak z^\dag}(x)^*d\m(x)\,,
\end{equation*}
with inverse
\begin{equation*}\label{cerigire}
\big(\mathscr F^{-1}_{\!\G,\mathfrak z^*_{\bu}}\,\Psi\big)(x):=2^d d!\int_{\mathfrak z^*_{\bu}}\!{\rm Tr}_{\xi_{\Z+\mathfrak z^\dag}}\big[\Psi(\Z)\xi_{\Z+\mathfrak z^\dag}(x)\big]\,|{\rm Pf}(\Z)|d\Z\,.
\end{equation*}

\begin{Remark}\label{oscura}
{\rm Setting
\begin{equation*}\label{em}
\mathscr V:=\mathscr F_{\G,\mathfrak z^*_\bu}\circ{\mathscr F}^{-1}_{\!\G,\g^*}\,,
\end{equation*}
one can write down explicit formulae, along the lines of Theorem \ref{harnarach}. Basically, $\mathscr V$ consists in sending $B\in\mathscr S(\g^*)$ into the family of restrictions $\big\{B|_{\Z+\mathfrak z^\dag}\,\big\vert\, \Z\in\o^\dag_\bu\equiv\mathfrak z^*_\bu\big\}$ and then into the family of operators 
$$
\big\{{\sf b}(\Z):={\sf Ped}_{\xi_{\Z+\z^\dag}}\!\big(B|_{\Z+\mathfrak z^\dag}\big)\,\big\vert\,\Z\in\o^\dag_\bu\big\}\,.
$$ 
Its inverse is given by
\begin{equation*}\label{gormula}
\big[\mathscr V^{-1}({\sf b})\big](\X)=\int_{\o}\!e^{i\<Y\mid\X\>}\,{\rm Tr}_\xi\big[{\sf b}(\Z)\xi_\Z(\exp Y)^*\big]d\lambda_\Z(Y)\,,\quad\forall\,\X\in\Z+\mathfrak z^\dag\!\subset\g^*.
\end{equation*}
}
\end{Remark}

A direct consequence is

\begin{Corollary}\label{similea}
Choosing the concrete option, one can also define a quantization 
\begin{equation*}\label{predictible}
\begin{aligned}
{\sf Op}_{\G\times\mathfrak z^*_\bu}&:\mathscr B^2(\G\times\mathfrak z^*_\bu\big)\to\mathbb B^2\big[L^2(\G)\big]\,,\\
\big[{\sf Op}_{\G\times\mathfrak z^*_\bu}(\Si)u\big](x)=&\,2^d d!\int_\G\int_{\mathfrak z^*_\bu}{\rm Tr}\big[\Xi(\mathcal Z)(y^{-1}x)\Si(x,\Z)\big]u(y)\,d\m(y)|{\rm Pf}(\Z)|d\Z\,.
\end{aligned}
\end{equation*}
This one is connected to the quantizations \eqref{tulcisor} and \eqref{ploicica} by
\begin{equation*}\label{chiueishen}
{\sf Op}_{\G\times\mathfrak z^*_\bu}[({\sf id}\otimes\mathscr M)(f)]={\sf Op}_{\G\times\wG}[({\sf id}\otimes\mathscr W)(f)]={\sf Op}_{\G\times\mathfrak g^*}(f)\,,\quad f\in\S(\G\times\g^*)\,,
\end{equation*}
\begin{equation*}\label{chiuretaj}
{\sf Op}_{\G\times\mathfrak z^*_\bu}[({\sf id}\otimes\tilde\Xi)(\si)]={\sf Op}_{\G\times\wG}(\si)\,,\quad \si\in\S(\G\times\wG)\,.
\end{equation*}
We recall that $\,\big[({\sf id}\otimes\tilde\Xi)(\si)\big](x,\Z)=\si\big(x,\xi_{\Z+\mathfrak z^\dag}\big)\,$ for every $x\in\G$ and $\Z\in\mathfrak z^*_\bu$\,.
\end{Corollary}

The next diagram tells us the story
$$
\begin{diagram}
\node{\mathscr B^2\!(\G\times\mathfrak z^*_\bu\big)}\arrow{se,b}{{\sf Op}_{\G\times\mathfrak z^*_\bu}}\node{\mathscr B^2(\G\times\wG)}\arrow{w,t}{{\sf id}\otimes\tilde\Xi} \arrow{s,r}{{\sf Op}_{\G\times\wG}}\node{L^2(\G\times\mathfrak g^*)}\arrow{w,t}{{\sf id}\otimes\mathscr W}\arrow{sw,r}{{\sf Op}_{\G\times\mathfrak g^*}}\\ 
\node{}\node{\mathbb B^2\!\big[L^2(\G)\big]}
\end{diagram}
$$

\section{Symbol classes for admissible graded groups}\label{hirkovan}

In this section we discuss symbol classes of H\"ormander's type in the case of graded nilpotent groups.

\subsection{Admissible graded Lie groups, their dilations and their Rockland operators}\label{futiceni}

We are going to review briefly some basic facts about graded Lie groups. Much more information can be found in \cite[Ch.\;4]{FR}; see also \cite{FS,Ri}.

\smallskip
Let $\G$ be a graded Lie group. Its Lie algebra can be written as a direct sum of vector subspaces
\begin{equation}\label{lie}
\g=\mathfrak w_1\oplus\dots\oplus \mathfrak w_l\,,
\end{equation}
where $\,[\mathfrak w_i,\mathfrak w_k]\subset\mathfrak w_{i+k}$ for every $i,k\in\{1,\dots,l\}$, and where $l$ is such that $\mathfrak w_{i+k}=\{0\}$ for $\,i+k> l$\,. Then $\G$ is a connected and simply connected nilpotent Lie group. Let us set
\begin{equation*}\label{dim}
n_k:=\dim\mathfrak w_k\,,\quad n:=\dim\g=n_1+n_2+\dots+n_l\,,
\end{equation*}
and define {\it the homogeneous dimension} 
\begin{equation*}\label{homdim}
Q:=n_1+2n_2+\dots+l n_l\,.
\end{equation*}

We are going to use bases $\{X_1,\dots,X_n\}$ of $\g$ such that for every $k$ the $n_k$ vectors $\big\{X_{j}\mid n_1+\dots+n_{k-1}<j\le n_1+\dots+n_{k-1}+n_k\big\}$ generate $\mathfrak w_k$  (we set $n_0=0$ for convenience)\,. 

\smallskip
The multiplicative group $(\R_+,\cdot)$ acts by automorphisms (called {\it dilations}) of the Lie algebra \eqref{lie} by
\begin{equation*}\label{dila}
\mathfrak{dil}_r\big(Y_{(1)},Y_{(2)},\dots,Y_{(l)}\big):=\big(r^{\nu_{1}} Y_{(1)},r^{\nu_{2}}Y_{(2)}\dots,r^{\nu_{l}} Y_{(l)}\big)\,,\quad r\in\R_+\,,\ Y_{(k)}\in\mathfrak w_k\,,\ 1\le k\le l\,.
\end{equation*}
One has 
\begin{equation*}\label{stoarpha}
\mathfrak{dil}_r(X_j)=r^{\nu_j} X_j\,,\quad r\in\R_+\,,\ 1\le j\le n\,,
\end{equation*}
in terms of the {\it dilation weights} $\,\nu_j$ if $\,n_1+\dots+n_{k-1}< j\le n_1+\dots+n_k$\,. The dilations on the dual $\g^*$ of the Lie algebra are defined by 
$$
\big\<X\!\mid\!\mathfrak{dil}^*_r(\mathcal X)\big\>:=\big\<\mathfrak{dil}_{r^{-1}}(X)\!\mid\!\mathcal X\big\>\,,\quad X\in\g\,,\mathcal X\in\g^*,\ r\in\R_+\,.
$$
Since the graded groups are exponential, one can apply dilations on the group side, setting
\begin{equation}\label{ansfer}
{\sf dil}_r(x):=\exp\big[\mathfrak{dil}_r({\log}\,x)\big]\,,\quad x\in\G\,,\ r\in\R_+\,.
\end{equation}
This induces unitary operators on $\H:=L^2(\G)$ by
\begin{equation}\label{sacele}
\big[{\sf Dil}(r)u\big](x):=r^{\frac{Q}{2}}\big(u\circ{\sf dil}_r\big)(x)=r^{\frac{Q}{2}}u\big({\sf dil}_r(x)\big)\,.
\end{equation}
Finally, using the duality between $\G$ and $\wG$ one makes $(\R_+,\cdot)$ act on the unitary dual by
$$
[\widehat{\sf dil}_r(\xi)](x):=\xi\big[{\sf dil}_{r^{-1}}(x)\big]\,,\quad x\in\G\,,\ \xi\in\wG\,,\ r>0\,.
$$
The deliberate confusion between irreducible representations and their unitary equivalence classes is convenient and harmless.

\smallskip
A {\it Rockland operator} is a (say left) invariant differential operator $\mathcal R$ on $\G$ such that, for every non-trivial irreducible representation $\xi:\G\to\mathbb B(\H_\xi)$\,, the operator $d\xi(\mathcal R)$  is injective on the subspace $\H_\xi^\infty$ of all smooth vectors. We prefer them to be homogeneous and positive: the homogeneity reads, using \eqref{sacele}
\begin{equation*}\label{brasov}
{\sf Dil}\big(r^{-1}\big)\mathcal R\,{\sf Dil}(r)=r^\nu\mathcal R\,,\quad\forall\,r\in\R_+\,.
\end{equation*}
The degree of homogeneity $\nu$ is a multiple of any of the dilation weights. A left invariant homogeneous differential operator is hypoelliptic if and only if it is a Rockland operator \cite{HN,FR}, see also \cite{FR1} for a detailed discussion.

\smallskip
{\bf Convention:} From now on we call simply {\it Rockland operator} a left invariant positive and homogeneous Rockland operator. These are the only ones appearing below; it is known that they exist on any graded group and, in fact, if such an operator exists on a connected simply connected Lie group, it has to be graded.

\begin{Remark}\label{cugris}
{\rm Concrete examples of homogeneous degree $q=2p$ are 
\begin{equation*}\label{ghimbav}
\mathcal R:=\sum_{j=1}^{n'}(-1)^{\frac{p}{\nu_j}}Z_j^{2\frac{p}{\nu_j}},
\end{equation*}
where $\{Z_j\}_{j=1,\dots,n}$ is a basis as in \cite[Lemma 3.1.14]{FR} (see also \cite{FR1}) and $p$ is a common multiple of the dilation weights. The basis is such that $Z_j$ is $\nu_j$-homogeneous, $Z_1,\dots,Z_{n'}$ generate $\g$ as a Lie algebra, while $Z_{n'+1},\dots,Z_n$ generate a vector space containing $[\g,\g]$\,.}
\end{Remark}

\begin{Remark}\label{stratificat}
{\rm Very concrete Rockland operators can be written down on {\it stratified groups}, which are graded groups for which $\mathfrak w_1$ in \eqref{lie} generates $\mathfrak g$ as a Lie algebra. If $\{X_1,\dots,X_{n_1}\}$ is a basis of the first stratum $\mathfrak w_1$\,, the left invariant $2$-homogeneous negative operator
$$
\mathcal L:=X_1^2+\dots+X_{n_1}^2
$$
is called {\it a sub-Laplacian}. Then $\mathcal R:=-\mathcal L$ is a Rockland operator. 
}
\end{Remark}

It is important to note that Rockland operators are decomposable. For the theory of unbounded decomposable operators we refer to \cite{DNSZ,Nu}. Let us fix a positive Rockland operator $\mathcal R$\,, homogeneous of order $\nu$\,. We also set $\,\mathcal T\!:=({\sf id}+\mathcal R)^{1/\nu}$\,.  The key fact is that both $\mathcal R$ and $\mathcal T$, acting in $L^2(\G)$\,, become decomposable operators in $L^2(\wG):=\int_{\wG}^\oplus\!\H_\xi d\wm(\xi)$ after applying unitary equivalence by an (extension of) the group Fourier transformation. Thus they are affiliated to the von Neumann algebra 
$$
\mathcal L^\infty(\wG):=\int_{\wG}^\oplus\!\mathbb B(\H_\xi) d\wm(\xi)\,,
$$
which is isomorphically represented as the left group von Neumann algebra
$$
\mathbb B_L\big[L^2(\G)]:=\big\{T\in \mathbb B\big[L^2(\G)]\mid T\ {\rm commutes\ with\ the\ left\ regular\ representation}\big\}.
$$
Setting 
\begin{equation*}\label{artdeco}
\mathscr R:=\mathscr F_{\!\G,\wG}\circ\mathcal R\circ\mathscr F^{-1}_{\!\G,\wG}\,,\quad\mathscr T:=\mathscr F_{\!\G,\wG}\circ\mathcal T\circ\mathscr F^{-1}_{\!\G,\wG}\,,
\end{equation*}
one has for instance $(\mathscr R\phi)(\xi)=\mathscr R(\xi)\phi(\xi)$\,, where it is known that $\mathscr R(\xi)=d\xi(\mathcal R)$ with domain $\H_\xi^\nu$\,, the represented Sobolev space of order $\nu$ defined at \cite[5.1.1]{FR}. Of course, each fiber operator acts continuously on the space of smooth vectors $d\xi(\mathcal R):\H^\infty_\xi\to\H_\xi^\infty$.
For $\mathscr T$ one has 
\begin{equation}\label{uttil}
\mathscr T(\xi)=\big({\sf id}_\xi+\mathscr R(\xi)\big)^{1/\nu}=\big({\sf id}_\xi+d\xi(\mathcal R)\big)^{1/\nu}\,.
\end{equation}

\subsection{The classes $S^m_{\rho,\delta}$}\label{hiranoian}

Let us fix $q\in C^\infty_{\rm pol}(\G):={\sf Log}\big[C^\infty_{\rm pol}(\g)\big]\subset\S^{\prime}(\G)$ (an intrinsic definition is also posible and other conditions or function spaces can be used). 
Then the operator of multiplication defined by ${\sf Mult}_q(u):=qu$ is well-defined linear and continuous in $\S(\G)$ and in $\S^{\prime}(\G)$\,. We also set 
$$
\Delta_q:={\mathscr F}_{\!\G,\wG}\circ{\sf Mult}_q\circ{\mathscr F}^{-1}_{\!\G,\wG}\in\mathbb B\big[\mathscr S(\wG)\big]\cap\mathbb B\big[\mathscr S^{\prime}(\wG)\big]
$$ 
and 
$$
\Gamma_q:={\mathscr F}_{\!\G,\g^*}\!\circ{\sf Mult}_q\circ{\mathscr F}^{-1}_{\!\G,\g^*}\in\mathbb B\big[\S(\g^*)\big]\cap\mathbb B\big[\S^{\prime}(\g^*)\big]\,.
$$
It follows directly from the definitions that
\begin{equation}\label{bueno}
\begin{aligned}
\Gamma_q&={\mathscr F}_{\!\G,\g^*}\circ{\sf Mult}_q\circ{\mathscr F}^{-1}_{\!\G,\g^*}\\
&=\big({\mathscr F}_{\!\G,\g^*}\circ{\mathscr F}^{-1}_{\!\G,\wG}\big)\circ\big({\mathscr F}_{\!\G,\wG}\circ{\sf Mult}_q\circ{\mathscr F}^{-1}_{\!\G,\wG}\big)\circ\big({\mathscr F}_{\!\G,\wG}\circ{\mathscr F}^{-1}_{\!\G,\g^*}\big)\\
&=\mathscr W^{-1}\!\circ\Delta_q\circ\mathscr W.
\end{aligned}
\end{equation}
It can also be shown that $\Gamma_q$ is the operator of (commutative) convolution with ${\mathscr F}_{\G,\g^*}\!(q)$ coming from the vector structure of the dual of the Lie algebra, but this will not be needed.

\smallskip
Certain special functions $q$ were used in \cite{FR,FR1} to express the symbol class conditions. For multi-indices $\alpha\in\N_0^n$\,, besides the usual length $|\alpha|:=\alpha_1+\alpha_2+\dots+\alpha_n$\,, one also uses {\it the homogeneous length} $[\alpha]:=\sum_{k=1}^n\nu_k\alpha_k$\,, in terms of the dilation exponents $\nu_1,\dots,\nu_n$\,. Also recall that a basis in $\G$ has been denoted by $\{X_1,\dots,X_n\}$\,, leading to the differential operators $X^\beta_x\equiv X^\beta\!:=X_1^{\beta_1}\dots X_n^{\beta_n}$. It is shown in \cite{FR,FR1} that for every $\alpha\in\N^n_0$ there exists a unique homogeneous polynomial $q_\alpha:\G\to\R$ of degree $[\alpha]$ such that
$$
\big(X^\beta q_\alpha\big)(\e)=\delta_{\alpha,\beta}\,,\quad\forall\,\alpha,\beta\in\N^n_0\,.
$$
These polynomials are involved in Taylor developments and useful in writing down asymptotic developments for the ${\sf Op}_{\G\times\wG}$ calculus. For $\alpha\in\mathbb N^n$, we set 
$$
\tilde q_\alpha(x):=q_\alpha(x^{-1})\,,\ \ \Delta^\alpha\!:=\Delta_{\tilde q_\alpha}\ \ {\rm and}\ \ \,\Gamma^\alpha\!:=\Gamma_{\tilde q_\alpha}\,.
$$

Assuming that the group $\G$ is both graded and admissible, let us fix a positive Rockland operator $\mathcal R$ homogeneous of degree $\nu$ and recall \eqref{uttil}. By $\sup_{\xi\in\wG}$ one denotes the essential supremum over $\wG$ with respect to the Plancherel measure. For fixed numbers $m\in\mathbb R$\,, $\delta\,,\rho\in[0,1]$ such that $\rho\le\delta$\,,
the classes $S^m_{\rho,\delta}(\G\times\wG)$ were defined in \cite{FR,FR1} by seminorm-conditions of the form
\begin{equation*}\label{strontiu}
\p\!\si\!\p_{S^m_{\rho,\delta;(\alpha,\beta,\gamma)}}\,:=\,\sup_{x\in\G}\,\sup_{\xi\in\wG}\,\Big\Vert\mathscr T(\xi)^{-m+\rho[\alpha]-\delta[\beta]+\gamma}\big[\big(X_x^\beta\otimes\Delta^\alpha\big)\si\big](x,\xi)\mathscr T(\xi)^{-\gamma}\,\Big\Vert_{\mathbb B(\H_\xi)}\!<\infty\,,
\end{equation*}
involving all the multi-indices $\alpha,\beta,\gamma\in\N^n$. 

\smallskip
We now write $\si=({\sf id}\otimes\mathscr W)f\in S^m_{\rho,\delta}(\G\times\wG)$ and try to see what the corresponding conditions on $f\in\mathscr S(\G\times\g^*)$ are. We recall that, by Corollary \ref{cercif}, one has $\big(({\sf id}\otimes\mathscr W)g\big)(x,\xi)={\sf Ped}_\xi\big(g|_{\{x\}\times\Omega_\xi}\big)$\,.  On the other hand, by \eqref{bueno},
\begin{equation*}\label{vals}
\big(X^\beta_x\otimes\Delta^\alpha\big)\circ({\sf id}\otimes\mathscr W)=({\sf id}\otimes\mathscr W)\circ\big(X^\beta_x\otimes\Gamma^\alpha\big)\,.
\end{equation*}
It follows immediately that
\begin{equation}\label{ztronziu}
\begin{aligned}
&\p\!f\!\p_{S^m_{\rho,\delta;(\alpha,\beta,\gamma)}(\G\times\g^*)}\,:=\,\p\!({\sf id}\otimes\mathscr W)f\!\p_{S^m_{\rho,\delta;(\alpha,\beta,\gamma)}}\\
&=\sup_{(x,\xi)\in\G\times\wG}\,\big\Vert\mathscr T(\xi)^{-m+\rho[\alpha]-\delta[\beta]+\gamma}\,{\sf Ped}_\xi\big[\big((X_x^\beta\otimes\Gamma^\alpha) f\big)\big\vert_{\{x\}\times\Omega_\xi}\big]\mathscr T(\xi)^{-\gamma}\Big\Vert_{\mathbb B(\H_\xi)}\,.
\end{aligned}
\end{equation}

\begin{Remark}\label{tofinish}
{\rm The spaces of symbols 
$$
S^m_{\rho,\delta}(\G\times\g^*):=({\sf id}\otimes\mathscr W)^{-1}S^m_{\rho,\delta}(\G\times\wG)\subset\mathscr S(\G\times\g^*)
$$ 
can be defined along these lines, and they play the same role for the quantization ${\sf Op}_{\G\times\g^*}$ as $S^m_{\rho,\delta}(\G\times\wG)$ played for the ${\sf Op}_{\G\times\wG}$\,-\,calculus in \cite{FR,FR1}. This relationship allows one to transfer all results known for ${\sf Op}_{\G\times\wG}$ to this setting. Hopefully, in a future paper, we are going to undertake the non-trivial task of  rephrasing the conditions \eqref{ztronziu} in a more tractable form, to write down explicit results for the pseudo-differential calculus on $\G\times\g^*$, to compare it with existing (but only left or right invariant) calculi and to apply it to some concrete problems.
}
\end{Remark}

\section{A four dimensional two-step stratified admissible group}\label{hirkan}

In this section we work out an example of a four dimensional two-step stratified admissible group and demonstrate the discussed quantizations in this setting.

\subsection{General facts}\label{horkan}

Let $\mathfrak g:=\R^4$ be the Lie algebra with the bracket
$$
\big[(q,p,s,t),(q',p',s',t')\big]:=\big(0,0,qp'\!-q'\!p,\delta(qp'\!-q'\!p)\big)\,,
$$
where $\delta\in\R$ is a real number. It is a four-dimensional step two Lie algebra with center 
$$
\mathfrak z=\{0\}\times\{0\}\times\R\times\R\,.
$$
The canonical basis being denoted by $\{Q,P,S,T\}$, the single nontrivial bracket is
$$
[Q,P]=S+\delta T.
$$

\begin{Remark}\label{stiulete}
{\rm 
\begin{enumerate}
\item[(i)]
For convenience, we flipped the variables. A direct correspondence with the notations of previous sections, where the central variables stay first, would require the transformation $(q,p,s,t)\to(s,t,q,p)$\,.
\item[(ii)] Here and subsequently, in vector spaces with specified basis one considers the Lebesgue measures canonically associated to these bases.
\end{enumerate}
}
\end{Remark}

The corresponding connected simply connected Lie group is $\G=\R^4$ with BCH-multiplication
$$
(q,p,s,t)\bu(q',p',s',t')=\Big(q+q',p+p',s+s'+\frac{1}{2}(qp'\!-q'\!p),t+t'+\frac{\delta}{2}(qp'\!-q'\!p)\Big)\,.
$$
The unit is $\,\mathbf 0:=(0,0,0,0)$ and the inverse of $(q,p,s,t)$ is $(-q,-p,-s,-t)$\,.
In this realization, the maps $\,\exp$ and $\,\log$ are simply the identity of $\R^4$. Clearly $\G$ is stratified with dilations
$$
(q,p,s,t)\to(rq,rp,r^2s,r^2t)\,,\quad r>0\,.
$$

\begin{Remark}\label{sanberg}
{\rm One has two short exact sequences:
$$
1\longrightarrow\R^2\equiv{\sf Z}\longrightarrow\G\overset{\Phi}{\longrightarrow}\R^2\longrightarrow 1\,,
$$
with $\Phi(q,p,s,t):=(q,p)$\,, and $2$-cocycle
$$
c_1:\R^2\times\R^2\to{\sf Z}\,,\quad c_1\big((q,p),(q'\!,p')\big):=\Big(\frac{1}{2}(qp'\!-q'\!p),\frac{\delta}{2}(qp'\!-q'\!p)\Big)\,,
$$
and
$$
1\longrightarrow\R\longrightarrow\G\overset{\Psi}{\longrightarrow}{\sf H}_1\longrightarrow 1\,,
$$
with $\Psi(q,p,s,t):=(q,p,s)$ and $2$-cocycle
$$
c_2:{\sf H}_1\times{\sf H}_1\to\R\,,\quad c_1\big((q,p,s),(q'\!,p'\!,s')\big):=\frac{\delta}{2}(qp'\!-q'\!p)\,.
$$
The second one presents our nilpotent Lie group as a central extension of the $3$-dimensional Heisenberg group by $\R$\,. It is split (actually a direct product) if and only if $\delta=0$\,.
}
\end{Remark}

\subsection{The coadjoint action}\label{horokan}

The adjoint action is
$$
\begin{aligned}
{\sf Ad}&_{(q_0,p_0,s_0,t_0)}(q,p,s,t)=(q_0,p_0,s_0,t_0)\bu(q,p,s,t)\bu(q_0,p_0,s_0,t_0)^{-1}\\
&=\Big(q_0+q,p_0+p,s_0+s+\frac{1}{2}(q_0p\!-qp_0),t_0+t+\frac{\delta}{2}(q_0p-qp_0)\Big)\bu(-q_0,-p_0,-s_0,-t_0)\\
&=\big(q,p,s+q_0p-qp_0,t+\delta(q_0p-qp_0)\big)\,.
\end{aligned}
$$
For the dual we set 
$\mathfrak g^*\!:=\R^4\ni(\rho,\vartheta,\varsigma,\tau)\equiv\rho\mathcal Q+\vartheta\mathcal P+\varsigma\mathcal S+\tau\mathcal T$
with duality
$$
\big\<(q,p,s,t)\!\mid\!(\rho,\vartheta,\sig,\tau)\big\>:=q\rho+p\vartheta+s\varsigma+t\tau.
$$
The anihilator of the center is 
$$
\mathfrak z^\perp=\R\times\R\times\{0\}\times\{0\}\,.
$$
The canonical bilinear form reads
\begin{equation}\label{fagana}
\begin{aligned}
{\rm Bil}_{(\rho,\vartheta,\varsigma,\tau)}\big((q,p,s,t),(q',p',s',t')\big)&=\big<\big[(q,p,s,t),(q',p',s',t')\big]\mid(\rho,\vartheta,\sig,\tau)\big>\\
&=(qp'\!-q'\!p)\sig+\delta(qp'\!-q'\!p)\tau,
\end{aligned}
\end{equation}
and it is non-degenerate when restricted to $\R^2\!\times\!\{(0,0)\}$\,.

\smallskip
Now we can compute the coadjoint action:
$$
\begin{aligned}
\Big\<(q,p,s,t)\,\big\vert\,&{\sf Ad}^*_{(q_0,p_0,s_0,t_0)}(\rho,\vartheta,\sig,\tau)\Big\>=\Big\<\sf Ad}_{(-q_0,-p_0,-s_0,-t_0)}(q,p,s,t)\,\big\vert\,{(\rho,\vartheta,\sig,\tau)\Big\>\\
&=\Big\<(q,p,s-q_0p+qp_0,t-\delta(q_0p-qp_0))\,\big\vert\,(\rho,\vartheta,\sig,\tau)\Big\>\\
&=q(\rho+p_0\sig+\delta p_0\tau)+p(\vartheta-q_0\varsigma-\delta q_0\tau)+s\sig+t\tau,
\end{aligned}
$$
meaning that
$$
\begin{aligned}
{\sf Ad}^*_{(q_0,p_0,s_0,t_0)}(\rho,\vartheta,\sig,\tau)&=\big(\rho+p_0\sig+\delta p_0\tau,\vartheta-q_0\sig-\delta q_0\tau,\sig,\tau\big)\\
&=\big(\rho,\vartheta,\sig,\tau\big)+\big(p_0[\sig+\delta\tau],-q_0[\sig+\delta\tau],0,0\big)\,.
\end{aligned}
$$

The fixed points have all the form
$$
(\rho,\vartheta,-\delta\tau,\tau)\,,\quad \rho,\vartheta,\tau\in\R\,.
$$
If $\sig\ne -\delta\tau$, the coadjoint orbit passing through $\big(\rho,\vartheta,\sig,\tau\big)$ is flat and $2$ - dimensional:
$$
\O_{(\rho,\vartheta,\sig,\tau)}={(\rho,\vartheta,\sig,\tau)}+\R^2\!\times\!\{(0,0)\}={\big(\rho,\vartheta,\sig,\tau\big)}+\mathfrak z^\perp={\big(0,0,\sig,\tau\big)}+\mathfrak z^\perp.
$$
It only depends on $(\sig,\tau)$ and can be written in the form
$$
\O_{(\sig,\tau)}={(0,0,\sig,\tau)}+\R^2\!\times\!\{(0,0)\}=\{(\rho,\vartheta,\sig,\tau)\mid \rho,\vartheta\in\R\}\,.
$$
The restriction of ${\sf Ad}^*_{(q_0,p_0,s_0,t_0)}$ to such an orbit is the translation by $(p_0\sig+\delta p_0\tau,-q_0\sig-\delta q_0\tau,0,0)$\,, so the invariant measures are all proportional to the transported $2$-dimensional Lebesgue measure $d\rho d\vartheta$ (cf. \cite[Th.\;1.2.12]{CG} for a more general statement)\,. So they can be written as $c(\sig,\tau)d\rho d\vartheta$ for positive numbers $c(\sig,\tau)$\,. The good normalisation, leading to the canonical measure $d\gamma_{\O_{(\sig,\tau)}}\!\equiv d\gamma_{(\sig,\tau)}$ used repeatedly above, is 
\begin{equation}\label{equivoc}
d\gamma_{(\sig,\tau)}(\rho,\vartheta)=(2|\sig+\delta\tau|)^{-1}d\rho d\vartheta\,.
\end{equation}
This can be seen rather easily by inspecting \cite[4.3]{CG}, but it is also explained in Remark \ref{referoaica}.

\smallskip
The isotropy group and algebra of the generic points $\big(\rho,\vartheta,\sig,\tau\big)$ are, respectively,
$$
\G_{(\rho,\vartheta,\sig,\tau)}=\{(0,0)\}\!\times\!\R^2,\quad\mathfrak g_{(\rho,\vartheta,\sig,\tau)}=\{(0,0)\}\!\times\!\R^2=\mathfrak z\,.
$$

According to the general theory or to \eqref{fagana}, the canonical bilinear form can be seen as a map from $\mathfrak z^*\cong\{(0,0)\}\!\times\!\R^2\equiv\R^2$ to the space of antisymmetric bilinear (non-degenerate) forms on the common pre-dual $\,\o=\R^2\!\times\!\{(0,0)\}\equiv\R^2$, given by
$$
{\rm Bil}_{(\sig,\tau)}\big((q,p),(q',p')\big)=(qp'\!-q'\!p)\sig+\delta(qp'\!-q'\!p)\tau.
$$
The determinant is $\,\det{\rm Bil}_{(\sig,\tau)}=(\sig+\delta\tau)^2$, so 
$$
{\rm Pf}(\si,\tau)=\sig+\delta\tau=0\quad\Longleftrightarrow\quad \sig=-\delta\tau,
$$
and the irreducible representations that are square integrable modulo the center (or, equivalently, the flat orbits) are labelled by 
$$
\mathfrak z^*_\bu=\{(\sig,\tau)\in\R^2\mid \sig\ne-\delta\tau\}\,.
$$
The transported Plancherel measure on $\mathfrak z^*$ is concentrated on this set and it has a density with respect to the $2$-dimensional Lebesgue measure:
$$
d\mu(\sig,\tau)=2|{\rm Pf}(\sig,\tau)|d\si d\tau=2|\sig+\delta\tau|d\sig d\tau.
$$ 

\begin{Remark}\label{referoaica}
{\rm One can combine this form of the Plancherel measure with Lemma \ref{adooa} to compute the canonical measures on our flat coadjoint orbits. Using the concrete realisation, \eqref{streeka} becomes
$$
\int_{\R^2}\Big[\int_{\O_{(\sig,\tau)}}\!\!C(\rho,\vartheta,\sig,\tau)d\gamma_{(\sig,\tau)}(\rho,\vartheta)\Big]2|\sig+\delta\tau|d\si d\tau=\int_{\mathfrak g^*}\!C(\rho,\vartheta,\sig,\tau)d\rho d\vartheta d\sig d\tau,
$$
from which \eqref{equivoc} follows.
}
\end{Remark}

\subsection{Irreducible representations}\label{hurkan}

The way to construct the irreducible representations of $\G$ is exposed in a general setting in \cite[Sect.\;2]{CG} and will be applied without many comments.

\smallskip
We first determine the irreducible representations attached to the fixed points. If $(\rho,\vartheta,-\delta\tau,\tau)$ is such a fixed point, the entire Lie algebra $\mathfrak g$ is polarizing (maximal subordinate). The associated character
$$
\chi_{(\rho,\vartheta,-\delta\tau,\tau)}:\G\equiv\mathfrak g\to\mathbb C\,,\quad\chi_{(\rho,\vartheta,-\delta\tau,\tau)}(q,p,s,t):=e^{i\<(q,p,s,t)\mid (\rho,\vartheta,-\delta\tau,\tau)\>}=e^{i(q\rho+p\vartheta+(t-\delta s)\tau)}
$$
is the representation we were looking for. We recall that these representations have no contribution to the Plancherel measure. If $u\in L^1(\G)$\,, then its group Fourier tranform computed in these characters (just complex numbers) can be expressed as a restriction of the $4$-dimensional Euclidean Fourier transform:
$$
\begin{aligned}
\big(\mathscr F_{\!\G,\wG}\,u\big)\big(\chi_{(\rho,\vartheta,-\delta\tau,\tau)}\big)&=\int_{\R^4}\!u(q,p,s,t)\chi_{(\rho,\vartheta,-\delta\tau,\tau)}(-q,-p,-s,-t)dqdpdsdt\\
&=\int_{\R^4}\!u(q,p,s,t)e^{-i(q\rho+p\vartheta+(t-\delta s)\tau)}dqdpdsdt\\
&=\big(\mathcal F_{\R^4}u\big)(\rho,\vartheta,-\delta\tau,\tau)\,.
\end{aligned}
$$

In search of the irreducible representations corresponding to the flat orbits, we fix the (Abelian) Lie subalgebra
$$
\mathfrak m:=\{0\}\times\R^3={\sf Span}(P,S,T)\,.
$$
It is clearly polarizing for all these flat orbits, since it has the right dimension $\,\dim\mathfrak m=\frac{1}{2}(\dim\mathfrak g+\dim\mathfrak z)$
and, by \eqref{fagana}, one has  
\begin{equation}\label{fagrana}
{\rm Bil}_{(\rho,\vartheta,\sig,\tau)}\big((0,p,s,t),(0,p',s',t')\big)=0\,.
\end{equation}
If $(\sig,\tau)\in\mathfrak z^*_\bu$\,, i.e. $\,\sig\ne -\delta\tau$, the character 
$$
\chi_{(\sig,\tau)}:{\sf M}\equiv\mathfrak m\to \T\,,\quad \chi_{(\sig,\tau)}(0,p,s,t)=e^{i\<(0,p,s,t)\mid (0,0,\sig,\tau)\>}=e^{i(s\sig+t\tau)}
$$
serves to induce the irreducible representation
$$
\pi_{(\sig,\tau)}:={\sf Ind}\big({\sf M}\!\uparrow\!\G;\chi_{(\sig,\tau)}\big):\G\to\mathbb B(\H_{(\sig,\tau)})\,.
$$
As model Hilbert space $\H_{(\si,\tau)}$ we are going to use $L^2(\G/{\sf M})\equiv L^2(\R)$\,. We will need the computation
$$
\begin{aligned}
(q_0,0,0,0)\bu(q,p,s,t)&=\Big(q_0+q,p,s+\frac{1}{2}q_0p,t+\frac{\delta}{2}p_0q\Big)\\
&=\Big(0,p,s+q_0p+\frac{1}{2}qp,t+\delta q_0p+\frac{\delta}{2}qp\Big)\bu\big(q_0+q,0,0,0\big)\,.
\end{aligned}
$$
Then the general theory gives for $\varphi\in L^2(\R)$ the expression of the corresponding irreducible representation
\begin{equation}\label{genery}
\begin{aligned}
\big[\pi_{(\sig,\tau)}(q,p,s,t)\varphi\big](q_0)&=\chi_{(\sig,\tau)}\Big(0,p,s+q_0p+\frac{1}{2}qp,t+\delta q_0p+\frac{\delta}{2}qp\Big)\varphi(q_0+q)\\
&=e^{i\big[s\sig+t\tau+\big(q_0p+\frac{1}{2}qp\big)(\sig+\delta\tau)\big]}\varphi(q_0+q)\,.
\end{aligned}
\end{equation}

\begin{Remark}\label{frivol}
{\rm Let $r>0$\,; denoting by ${\sf dil}_{\sqrt r}$ the unitary dilation operator $\big({\sf dil}_{\sqrt r}\,\varphi\big)(q_0):=\sqrt r\varphi(\sqrt rq_0)$ in $L^2(\R)$\,, one checks easily that
$$
\pi_{(r\sig,r\tau)}(q,p,s,t)\circ{\sf dil}_{\sqrt r}={\sf dil}_{\sqrt r}\circ\pi_{(\sig,\tau)}(\sqrt r q,\sqrt r p,rs,rt)={\sf dil}_{\sqrt r}\circ\pi_{(\sig,\tau)}[\sqrt r\!\cdot\!(q,p,s,t)]\,.
$$
}
\end{Remark}

Setting $\mathfrak q$ for the operator of multiplication with the variable in $L^2(\R)$\,, i.e. $(\mathfrak q\varphi)(q):=q\varphi(q)$\,, the four infinitesimal generators are
$$
d\pi_{(\sig,\tau)}(Q)=\partial,\quad d\pi_{(\sig,\tau)}(P)=i(\sig+\delta\tau)\mathfrak q\,,\quad d\pi_{(\sig,\tau)}(S)=i\si{\sf Id}\,,\quad d\pi_{(\sig,\tau)}(T)=i\tau{\sf Id}\,.
$$
Thus the repesented version of the canonical sub-Laplacian $\mathcal L:=Q^2+P^2$ is
$$
d\pi_{(\sig,\tau)}(\mathcal L)=\partial^2-(\sig+\delta\tau)^2\mathfrak q^2.
$$
The general theory tells us that $\,\H^\infty_{(\sig,\tau)}\!=\mathcal S(\R)$\,.

\subsection{The group Fourier transform and the Pedersen calculus}\label{uragan}

We are going to make use of Weyl's quantization \cite{Fo} with a parameter $\lambda\in\R_\bu:=\R\!\setminus\!\{0\}$ in one dimension
$$
[{\sf Weyl}_\lambda(\gamma)\varphi](q_0):=\int_{\R}\!\int_{\R}e^{i(q_0-q)\eta}\gamma\Big(\eta,\lambda\frac{q_0+q}{2}\Big)\varphi(q) dqd\eta\,.
$$

In computing the group Fourier transform and the Pedersen quantization, we only treat the (generic) flat orbits $\,\O_{(\sig,\tau)}=(0,0,\sig,\tau)+\mathbb R^2\!\times\!\{(0,0)\}$\,, where $\sig\ne-\delta\tau$. One denotes by $\mathcal F_{\R^k}$ the usual Euclidean Fourier transform in $\R^k$.

\begin{Proposition}\label{mix}
If $\,\lambda:=\sig+\delta\tau\ne 0$ and (say) $u\in\S(\G)$\,, then $\big(\mathscr F_{\!\G,\wG}\,u\big)\big(\pi_{(\sig,\tau)}\big)$ is an integral operator in $L^2(\R)$ with kernel
$$
\kappa^u_{(\sig,\tau)}(q_0,q):=\big[({\sf Id}\otimes\mathcal F_{\mathbb R^3})u\big]\Big(q_0-q,\frac{\lambda}{2}(q_0+q),\sig,\tau\Big)
$$
and also a Weyl $\lambda$-pseudo-differential operator with symbol $\big(\mathcal F_{\mathbb R^4}u\big)|_{\R^2\times\{(\sig,\tau)\}}$\,, i.e. it is given by
$$
\big[\big(\mathscr F_{\!\G,\wG}\,u\big)\big(\pi_{(\sig,\tau)}\big)\varphi\big](q_0)=
\int_\R\int_\R e^{i(q_0-q)\eta}\big(\mathcal F_{\mathbb R^4}u\big)\Big(\eta,\frac{\sig+\delta\tau}{2}(q_0+q),\sig,\tau\Big)\varphi(q)dq.
$$
\end{Proposition}

\begin{proof}
If $u\in L^1(\G)$ and $\varphi\in L^2(\R)$\,, then
$$
\begin{aligned}
\big[\big(\mathscr F_{\!\G,\wG}\,u\big)&\big(\pi_{(\sig,\tau)}\big)\varphi\big](q_0)=\int_{\R^4}\!u(q,p,s,t)\big[\pi_{(\sig,\tau)}(-q,-p,-s,-t)\varphi\big](q_0)dqdpdsdt\\
&=\int_{\R^4}\!u(q,p,s,t)e^{-i\big[s\sig+t\tau+\big(q_0p-\frac{1}{2}qp\big)(\sig+\delta\tau)\big]}\varphi(q_0-q)dqdpdsdt\\
&=\int_\R\big[({\sf Id}\otimes\mathcal F_{\mathbb R^3})u\big]\Big(q_0-q,\frac{\sig+\delta\tau}{2}(q_0+q),\si,\tau\Big)\varphi(q)dq\\
&=\int_\R\int_\R e^{i(q_0-q)\eta}\big(\mathcal F_{\mathbb R^4}u\big)\Big(\eta,\frac{\sig+\delta\tau}{2}(q_0+q),\sig,\tau\Big)\varphi(q)dq\,,
\end{aligned}
$$
finishing the proof.
\end{proof}

\begin{Proposition}\label{desen}
If $\,(\sig,\tau)\in\mathfrak z^*_\bu$\,, then $\,{\sf Ped}_{\O_{(\sig,\tau)}}\!\equiv{\sf Ped}_{(\sig,\tau)}$ only depends on the combination $\lambda:=\sig+\delta\tau\ne 0\,$ and one has 
$$
{\sf Ped}_{(\sig,\tau)}=(2|\sig+\delta\tau|)^{-1}\,{\sf Weyl}_{\sig+\delta\tau}\,.
$$ 
\end{Proposition}

\begin{proof}
We start with the Fourier transform adapted to the coadjoint orbit. For any $\Psi\in\S\big(\O_{(\sig,\tau)}\big)$\,, seen as a function of $(\rho,\vartheta)\in\R^2$\,, and for any $(q,p,s,t)\in\mathfrak g$\,, we have by \eqref{equivoc} that
\begin{equation*}\label{incep}
\hat{\Psi}(q,p,s,t):=\int_{\O_{(\sig,\tau)}}\!\!\!e^{-i\<(q,p,s,t)\mid (\rho,\vartheta,\sig,\tau)\>}\Psi(\rho,\vartheta)\,(2|\sig+\delta\tau|)^{-1}d\rho d\vartheta
\end{equation*}
and then, since the common predual is $\,\o=\R^2\!\times\!\{(0,0)\}$\,,
\begin{equation*}\label{inchep}
\big[{\sf F}_{\O_{(\sig,\tau)}}(\Psi)\big](q,p)=\hat{\Psi}(q,p,0,0)=(2|\sig+\delta\tau|)^{-1}\int_{\R^2}\!\!e^{-i(q\rho+p\vartheta)}\Psi(\rho,\vartheta)\,d\rho d\vartheta\,,
\end{equation*}
so (after some identifications), we essentially arrived once more at the Euclidean Fourier transform. 
Then, taking into account \eqref{genery} and the definition of the Pedersen quantization, for $\Psi\in\S\big(\O_{(\sig,\tau)}\big)$ one can write
$$
\begin{aligned}
\big[{\sf Ped}_{(\sig,\tau)}(\Psi)&\varphi](q_0)=(2|\sig+\delta\tau|)^{-1}\!\int_{\R^2}\!\int_{\R^2}e^{-i(q\rho+p\vartheta)}\Psi(\rho,\vartheta)\big[\pi_{(\sig,\tau)}(q,p,0,0)\varphi](q_0) dqdpd\rho d\vartheta\\
&=(2|\sig+\delta\tau|)^{-1}\!\int_{\R^2}\!\int_{\R^2}e^{-i(q\rho+p\vartheta)}\Psi(\rho,\vartheta)e^{i\big(q_0p+\frac{1}{2}qp\big)(\sig+\delta\tau)}\varphi(q_0+q) dpdqd\rho d\vartheta\\
&=|2\lambda|^{-1}\!\int_{\R^3}\!\Big(\int_{\R}e^{ip\big[\lambda\big(q_0+\frac{1}{2}q\big)-\vartheta\big]}dp\Big)e^{-iq\rho}\,\Psi(\rho,\vartheta)\varphi(q_0+q) dqd\rho d\vartheta\\
&=|2\lambda|^{-1}\!\int_{\R}\!\int_{\R}e^{-iq\rho}\,\Psi\big(\rho,\lambda(q_0+q/2)\big)\varphi(q_0+q) dpd\rho\\
&=|2\lambda|^{-1}\!\int_{\R}\!\int_{\R}e^{i(q_0-q)\rho}\,\Psi\Big(\rho,\lambda\frac{q_0+q}{2}\Big)\varphi(q) dqd\rho\,.
\end{aligned}
$$
The forth equality is a formal but easy to justify standard fact.
\end{proof}

\subsection{Quantization}\label{hurgan}

Therefore, as in Corollary \ref{similea}, the concrete form of the global group quantization is
\begin{equation}\label{pterodactil}
\begin{aligned}
&\big[{\sf Op}_{\G\times\mathfrak z^*_\bu}(\Si)u\big]\big(q',p',s',t'\big)\\
=2\int_{\R^4}\!\int_{\R^2}\,&{\rm Tr}\Big[\pi_{(\sig,\tau)}\Big(q'\!-q,p'\!-p,s'\!-s-\frac{1}{2}(qp'\!-q'\!p),t'\!-t-\frac{\delta}{2}(qp'\!-q'\!p)\Big)
\Si\big(q',p',s',t';\sig,\tau\big)\Big]\\
&u(q,p,s,t)|\sig+\delta\tau|d\sig d\tau\, dqdpdsdt\,.
\end{aligned}
\end{equation}

On the other hand, the quantization on $\G\times\mathfrak g^*$ indicated in \eqref{tulcisor} reads now
\begin{equation}\label{melcisor}
\begin{aligned}
\big[{\sf Op}_{\G\times\g^*}&\!(f)u\big]\big(q',p',s',t'\big)=\int_{\R^4}\int_{\R^4}\,e^{i[(q'\!-q)\rho+(p'\!-p)\vartheta+(s'\!-s)\sig+(t'\!-t)\tau]}\,e^{-\frac{i}{2}(qp'\!-q'\!p)(\si+\delta\tau)}\\
&f\big(q',p',s',t';\rho,\vartheta,\sig,\tau\big)u(q,p,s,t)dqdpdsdt\,d\rho d\vartheta d\sig d\tau.
\end{aligned}
\end{equation}

\begin{Remark}
{\rm By \eqref{aigrija} and \eqref{equivoc}, one has
\begin{equation*}\label{pterotactil}
\begin{aligned}
&{\rm Tr}\Big[\pi_{(\si,\tau)}\Big(q'\!-q,p'\!-p,s'\!-s-\frac{1}{2}(qp'\!-q'\!p),t'\!-t-\frac{\delta}{2}(qp'\!-q'\!p)\Big){\sf Ped}_{(\sig,\tau)}\Big(f\big(q',p',s',t';\cdot,\cdot,\si,\tau\big)\Big)\Big]\\
&=(2|\si+\delta\tau|)^{-1}\!\int_{\R^2}e^{i\big\<q'\!-q,p'\!-p,s'\!-s-\frac{1}{2}(qp'\!-q'\!p),t'\!-t-\frac{\delta}{2}(qp'\!-q'\!p)\,\mid\,(\rho,\vartheta,\sig,\tau)\big\>}f\big(q',p',s',t';\rho,\vartheta,\si,\tau\big)d\rho d\vartheta\\
&=(2|\si+\delta\tau|)^{-1}\!\int_{\R^2}e^{i[(q'\!-q)\rho+(p'\!-p)\vartheta+(s'\!-s)\si+(t'\!-t)\tau]}\,e^{-\frac{i}{2}(qp'\!-q'\!p)(\sig+\delta\tau)}f\big(q',p',s',t';\rho,\vartheta,\si,\tau\big)d\rho d\vartheta\,.
\end{aligned}
\end{equation*}
Replacing this in \eqref{pterodactil}, for 
$$
\Si\big(q',p',s',t';\si,\tau\big)={\sf Ped}_{(\si,\tau)}\Big(f\big(q',p',s',t';\cdot,\cdot,\si,\tau\big)\Big)\,,
$$
one recovers \eqref{melcisor}. This is a confirmation of Corollary \ref{cercif} in this simple particular case. 
}
\end{Remark}

\section{Appendix: Admissible graded Lie algebras with one-dimensional center}\label{hirokorkan}

Here we discuss several examples of admissible graded Lie groups and the appearing elements of their representatons.

\subsection{Automorphisms and Lie algebras with one-dimensional center}\label{hirocin}

An automorphism of the Lie algebra generates many other automorphisms that interact well  with the coadjoint picture. Let $\mathfrak c:\g\to\g$ be such an automorphism; then another one is defined by 
$$
\mathfrak c^*:\g^*\to\g^*,\quad\big\<X\!\mid\!\mathfrak c^*(\X)\big\>:=\big\<\mathfrak c^{-1}(X)\!\mid\!\X\big\>\,.
$$
The exponential map being an isomorphism, one can apply the automorphism on the group side, setting
\begin{equation}\label{lansfer}
{\sf c}(x):=\exp\!\big[\mathfrak{c}({\log}\,x)\big]\,,\quad x\in\G\,.
\end{equation}
Finally, using the duality between $\G$ and $\wG$ one makes the automorphism act on the unitary dual by
$$
[\widehat{\sf c}(\xi)](x):=\xi\big[{\sf c}(x)\big]\,,\quad x\in\G\,,\ \xi\in\wG\,.
$$
The deliberate confusion between irreducible representations and their unitary equivalence classes is harmless. If we adopt the representation point of view, note that the Hilbert spaces of $\xi$ and $\widehat{\sf c}(\xi)$ are the same.

\begin{Lemma}\label{potrivire}
Let $\mathfrak c$ be an automorphisms of the Lie algebra $\g$\,. 
\begin{enumerate}
\item[(i)]
For every $x\in\G$ one has
\begin{equation}\label{epuizzant}
{\sf Ad}_x\circ\mathfrak c=\mathfrak c\circ{\sf Ad}_{{\sf c}^{-1}(x)}\,,\quad\mathfrak{c}^*\!\circ{\sf Ad}_x^*={\sf Ad}_{{\sf c}^{-1}(x)}^*\!\circ\mathfrak{c}^*.
\end{equation}
\item[(ii)]
The automorhism $\mathfrak c^*$ sends coadjoint orbits in coadjoint orbits.
\item[(iiii)]
If $\,\mathfrak m$ is a polarizing subalgebra for $\U\in\g$\,, then $\mathfrak c(\mathfrak m)$ is a polarizing subalgebra for $\mathfrak c^*(\U)$\,.
\item[(iv)]
If $\,\G$ is admissible, the automorhism $\mathfrak c^*$ sends flat coadjoint orbits in flat coadjoint orbits.
\end{enumerate}
\end{Lemma}

\begin{proof}
(i) Using \eqref{lansfer} and notations from Subsection \ref{Perdaf}, one has
$$
\begin{aligned}
{\sf Ad}_x\circ\mathfrak c&=\log\circ\,{\rm inn}_x\!\circ\exp\circ\,\mathfrak c=\log\circ\,{\rm inn}_x\!\circ{\sf c}\circ\exp\\
&=\log\circ\,{\sf c}\circ{\rm inn}_{{\sf c}^{-1}(x)}\!\circ\exp=\mathfrak c\circ\log\circ\,{\rm inn}_{{\sf c}^{-1}(x)}\!\circ\exp\\
&=\mathfrak c\circ{\sf Ad}_{{\sf c}^{-1}(x)}\,,
\end{aligned}
$$
which shows the first identity in \eqref{epuizzant}. The second one follows by duality.

\smallskip
(ii)
The second identity in \eqref{epuizzant} implies immediately that 
\begin{equation}\label{descoperita}
\mathfrak{c}^*\big(\O_{\mathcal U}\big)=\O_{\mathfrak{c}^*(\mathcal U)}\,,\quad \forall\;\mathcal U\in\g^*.
\end{equation}

(iii)
Straightforward proof: see \cite[Prop.\;1.3.6]{CG}.

\smallskip
(iv)
Recall that the flat orbits are of the form $\O=\mathcal Z+\z^\perp$, with $\,\mathcal Z\in\z^*_\bu$\,. So, by \eqref{descoperita}, it is enough to show that $\mathfrak c^*:\g^*\to\g^*$ leaves $\z^*_\bu$ invariant.
Clearly $\z^*$ is invariant under the automorphism ${\mathfrak c}^*$: use for instance the fact that the center $\z$ is invariant under any automorphism of $\g$\,. The points $\mathcal Z\in\z^*_\bu$ are characterized by the condition ${\rm Pf}(\mathcal Z)^2\!=\det\big({\rm Bil}_\mathcal Z\big)\ne 0$\,. But 
$$
{\rm Bil}_{\mathfrak{c}^*\!(\mathcal Z)}={\rm Bil}_{\mathcal Z}\circ\big(\mathfrak{c}^{-1}\!\times\mathfrak{c}^{-1}\big)
$$
and this implies the invariance of  $\z^*_\bu$\,. 
\end{proof}

The point (ii) tells us that we have a well-defined bijection $\widetilde{\mathfrak c^*}:\g^*/{\sf Ad^*}\to\g^*/{\sf Ad^*}$.
It can be shown that this map is compatible with the one acting on the level of the unitary dual:
$$
\widehat{\sf c}(\xi_\O)=\xi_{\tilde{\mathfrak c}^*\!(\O)}\,,\quad\forall\,\O\in\g^*/{\sf Ad^*}.
$$
This combined with (iv), or a direct proof, shows in the admissible case that $\,\widehat{\sf c}(\wG_{\bu})\subset\wG_{\bu}$\,.

\begin{Remark}\label{revin}
{\rm Having the form $\O_{\mathcal U}=\mathcal U+\z^\dag$ for some $\U\in\z^*_\bu$\,, the flat orbits can be obtained from each other through translations. But these translations in $\g^*$ are not corresponding to Lie algebra automorphisms of $\g$ and they do not match our setting.
}
\end{Remark}

If we know in advance that two irreducible representations $\xi$ and $\eta$ are connected by an automorphism, i.\;e. $\,\eta=\widehat{\sf c}(\xi)$\,, this is valuable: they both act on the same Hilbert space, and one can be easily constructed in terms of the other. Consequently, the Pedersen quantizations ${\sf Ped}_\xi$ and ${\sf Ped}_\eta$ are also directly connected. This is particularly effective when an automorphism group acts transitively on a relevant family of (classes of) irreducible representations. Under favorable circumstances, this can be used in the framework of the quantizations we studied on $\G\times\wG$ and $\G\times\g^*$ respectively. 

\smallskip
Let us just explore {\it the case when $\G$ is admissible and graded and the center $\z$ of the Lie algebra is one-dimensional}. This happens for the Heisenberg groups and for the examples in Subsections \ref{hirokan} and \ref{altkeva}, but it does not hold in Section \ref{hirkan}. When this is the case, then $\z^*$ is also one-dimensional and, since the Pfaffian is a homogeneous polynomal, it is easy to see that $\z^*_\bu=\z^*\!\setminus\!\{0\}$\,. The dilation group on $\g$ generates, as above, groups of dilations on $\g^*,\g^*/{\sf Ad}^*,\G,\wG,\wG_\bu,\z^*$ and $\z^*_\bu$\,. In particular, in $\z^*_\bu\cong\R\!\setminus\!\{0\}$ there are two orbits $\R_{\pm}$\,. Another automorphism $\mathfrak{inv}(X):=-X$ (or, equivalently, ${\sf inv}(x)=x^{-1}$ at the group level) connects the two orbits, because it acts as $\mathfrak{inv}^*(\mathcal Z)=-\mathcal Z$ on the dual of the center of $\g$\,.
Thus $1\in\R\!\setminus\!\{0\}$ can be connected with any other element in $\mathfrak z^*_\bu$ and one has 
\begin{equation}\label{liabel}
\xi_r=\xi_1\circ{\sf dil}_r\ \ {\rm if}\ \ r>0\quad {\rm and}\quad \xi_r=\xi_1\circ{\sf dil}_r\circ{\sf inv}\ \ {\rm if}\ \ r<0\,.
\end{equation}
Consequently, in this case, if one of the generic irreducible representation (corresponding to one of the flat orbits) is computed, the others are easily generated using the dilations and eventually an inversion. We recall that, by abuse, representations has been identified with their unitary equivalence classes; thus, in terms of representations, \eqref{liabel} merely means equivalence and not equality.

\subsection{A five dimensional three-step graded admissible group}\label{hirokan}

For a first example \cite[Ex.\;5.7]{BB5}, the Lie algebra is generated by $\{E_0,E_1,E_2,E_3,E_4\}$\,, with the non-trivial brackets
$$
[E_4,E_1]=[E_3,E_2]=E_0\,,\quad[E_4,E_3]=E_1\,.
$$
So it can be seen as $\R^5$ with bracket
$$
\big[(q_0,q_1,q_2,q_3,q_4),(p_0,p_1,p_2,p_3,p_4)\big]:=(q_4p_1-q_1p_4+q_3p_2-q_2p_3,q_4p_3-q_3p_4,0,0,0)\,.
$$
One has $[\g,\g]=\R^2\times\{(0,0,0)\}\equiv{\rm Span}\{E_0,E_1\}$ and the center is one-dimensional:
$$
\mathfrak z=\big[[\g,\g],\g\big]=\R\times\{(0,0,0,0)\}\equiv{\rm Span}\{E_0\}\,.
$$
The Lie algebra $\g$ is graded by
$$
\g=\mathfrak w_1\!\oplus\mathfrak w_2\!\oplus\mathfrak w_3={\rm Span}\{E_3,E_4\}\oplus{\rm Span}\{E_1,E_2\}\oplus{\rm Span}\{E_0\}\,,
$$
so the dilations are
$$
\mathfrak{dil}_r(q_0,q_1,q_2,q_3,q_4):=\big(r^3q_0,r^2q_1,r^2q_2,rq_3,rq_4\big)\,,\quad r\in\R_+\,.
$$
Since this Lie algebra is not stratified, there is no sub-Laplacian to use. A Rockland operator can be computed by applying Remark \ref{cugris}. Our basis satisfies the assumptions, since $\{E_2,E_3,E_4\}$ generates $\G$ as a Lie algebra and $\{E_0,E_1\}$ generates $[\g,\g]$ linearly; thus $n=5$ and $n'\!=3$\,. One may take $p=6$ and
$$
\mathcal R:=-E_2^6+E_3^{12}+E_4^{12}
$$
is a $12$-homogeneous Rockland operator.

\smallskip
For the dual we set 
$$
\mathfrak g^*\!:=\R^5\ni(\rho_0,\rho_1,\rho_2,\rho_3,\rho_4)\equiv\rho_0\mathcal E_0+\rho_1\mathcal E_1+\rho_2\mathcal E_2+\rho_3\mathcal E_3+\rho_4\mathcal E_4\,.
$$
The anihilator of the center is 
$$
\mathfrak z^\perp=\{0\}\times\R^4={\rm Span}\{\mathcal E_1,\mathcal E_2,\mathcal E_3,\mathcal E_4\}\,,
$$
while the dual of the center identifies to
$$
\z^*\equiv\R\times\{(0,0,0,0)\}={\rm Span}\{\mathcal E_0\}\,.
$$

On the corresponding connected simply connected Lie group $\G\equiv\g=\R^4$ one has the multiplication
$$
\begin{aligned}
&(q_0,q_1,q_2,q_3,q_4)\bu(p_0,p_1,p_2,p_3,p_4)\\
&=\Big(q_0+p_0+\frac{1}{2}(q_4p_1\!-q_1p_4+q_3p_2\!-q_2p_3)+\frac{1}{12}(q_4\!-p_4)(q_4p_3\!-q_3p_4),\\
&\quad\quad q_1+p_1+\frac{1}{2}(q_4p_3\!-q_3p_4),q_2+p_2,q_3+p_3,q_4+p_4\Big)\,.
\end{aligned}
$$

One easily computes the coadjoint action
$$
\begin{aligned}
&{\sf Ad}^*_{(q_0,q_1,q_2,q_3,q_4))}(\rho_0,\rho_1,\rho_2,\rho_3,\rho_4)\\
&=\big(\rho_0,\rho_1-q_4\rho_0,\rho_2-q_3\rho_0,\rho_3+(q_2-q_4^2/4)\rho_0-q_4\rho_1,\rho_4+(q_1+(1/4)q_4q_3)\rho_0+q_3\rho_1\big)\,.\\
\end{aligned}
$$

The fixed points ($0$-dimensional coadjoint orbits) are those situated in the subspace $\,\{(0,0)\}\times\R^3$.

\smallskip
Other ($2$-dimensional) coadjoint orbits are 
$$
(0,\rho_1,0,0,0)+\{(0,0,0)\}\times\R^2\,,\quad \rho_1\in\R\setminus\{0\}\,.
$$

The flat (generic, $4$-dimensional) coadjoint orbits have all the form 
$$
\O_{(\rho_0,0,0,0,0)}\equiv\O_{\rho_0}\!:=\{\rho_0\}\times\R^4=\{(\rho_0,0,0,0,0)\}+\{0\}\times\R^4=\{(\rho_0,0,0,0,0)\}+\z^\perp\,
$$ 
for some fixed $\rho_0\ne 0$\,. 

\smallskip
The canonical bilinear form reads
\begin{equation*}\label{fagania}
\begin{aligned}
{\rm Bil}_{(\rho_0,\rho_1,\rho_2,\rho_3,\rho_4)}&\big((q_0,q_1,q_2,q_3,q_4),(p_0,p_1,p_2,p_3,p_4)\big)\\
=&(q_4p_1\!-q_1p_4+q_3p_2\!-q_2p_3)\rho_0+(q_4p_3-q_3p_4)\rho_1
\end{aligned}
\end{equation*}
and is non-degenerate when restricted to the common predual $\,\o=\{0\}\times\R^4\subset\g$ of the flat orbits if and only if $\rho_0\ne 0$\,.
 Now, if $(\rho_0,0,0,0,0)\in\z^*$, one has
 $$
{ \rm Bil}_{(\rho_0,0,0,0,0)}\big((q_1,q_2,q_3,q_4),(p_1,p_2,p_3,p_4)\big)=(q_4p_1\!-q_1p_4+q_3p_2\!-q_2p_3)\rho_0\,,
 $$
 so
 $$
 {\rm Pf}^{2}(\rho_0)\equiv{\rm Pf}^2(\rho_0,0,0,0,0)=\det{ \rm Bil}_{(\rho_0,0,0,0,0)}=\rho_0^4\,,
 $$
 confirming once again that 
 $$
 \z^*_\bu=\R\!\setminus\!\{0\}\equiv(\R\!\setminus\!\{0\})\times\{(0,0,0,0)\}\,.
 $$
 Based on Proposition \ref{aiureala}, the concrete Plancherel measure on $\z^*_\bu$ is
 $$
 d\mu(\rho_0)=8|{\rm Pf}(\rho_0)|d\rho_0=8|\rho_0|^2d\rho_0\,.
 $$

\subsection{Another five dimensional three-step graded admissible group}\label{altkeva}

We present briefly a similar case, that is mentioned in \cite[Ex.\;5.8]{BB5}; it is different, slightly more complicated, but still similar to the one treated above. The non-trivial structure equations are
$$
[E_4,E_3]=E_2\,,\quad[E_4,E_2]=E_1\,,\quad[E_4,E_1]=[E_3,E_2]=E_0\,,
$$
corresponding to the bracket
$$
\big[(q_0,q_1,q_2,q_3,q_4),(p_0,p_1,p_2,p_3,p_4)\big]:=(q_3p_2-q_2p_3+q_4p_1-q_1p_4,q_4p_2-q_2p_4,q_4p_3-q_3p_4,0,0)\,.
$$
The center $\z={\rm Span}(E_0)$ is once again one dimensional, so the relevant square integrable modulo the center irreducible representations are generated by dilations and the inversion from a given one; we leave their computation to the reader. The dilation is defined by the (non-stratified) grading 
$$
\g={\rm Span}\{E_4\}\oplus{\rm Span}\{E_3\}\oplus{\rm Span}\{E_2\}\oplus{\rm Span}\{E_1\}\oplus{\rm Span}\{E_0\}\,,
$$
i.e. 
$$
\mathfrak{dil}_r(E_j):=r^{5-j}E_j\,,\quad r>0\,,\ j=0,1,2,3,4\,.
$$
Applying Remark \ref{cugris} with $\,n=5\,,n'=2$\,, $Z_j=E_{5-j}$ and $p=3\cdot 4\cdot 5=60$\,, one checks easily that
$$
\mathcal R:=E_4^{120}+E_3^{60}
$$
is a homogeneous Rockland operator of order $120$\,.

\smallskip
It is not difficult to see that the flat coadjoint orbits are labelled by $\rho_0\ne 0$\,:
$$
\O_{\rho_0}=\{\rho_0,0,0,0,0\}+\z^\perp=\{(\rho_0,\rho_1,\rho_2,\rho_3,\rho_4)\mid \rho_1,\rho_2,\rho_3,\rho_4\in\R\}\,.
$$
The canonical bilinear form
\begin{equation*}\label{fagaraniana}
\begin{aligned}
{\rm Bil}_{(\rho_0,\rho_1,\rho_2,\rho_3,\rho_4)}&\big((q_0,q_1,q_2,q_3,q_4),(p_0,p_1,p_2,p_3,p_4)\big)\\
=&(q_3p_2-q_2p_3+q_4p_1-q_1p_4)\rho_0+(q_4p_2-q_2p_4)\rho_1+(q_4p_3-q_3p_4)\rho_2
\end{aligned}
\end{equation*}
is most relevant for $\mathcal Z=(\rho_0,0,0,0,0)\in\z^*$ and $X=(0,q_1,q_2,q_3,q_4),Y=(0,p_1,p_2,p_3,p_4)\in\o$ (the commun predual of all the flat coadjoint orbits), leading to
\begin{equation*}\label{fagaranna}
{\rm Bil}_{\rho_0}\big((q_0,q_1,q_2,q_3,q_4),(p_0,p_1,p_2,p_3,p_4)\big)\\
=(q_3p_2-q_2p_3+q_4p_1-q_1p_4)\rho_0\,.
\end{equation*}
So, as in Subsection \ref{hirokan}\,, the Plancherel measure on $\z^*_\bu\equiv\R\!\setminus\!\{0\}$ is
$$
d\mu(\rho_0)=8|{\rm Pf}(\rho_0)|d\rho_0=8\det{ \rm Bil}_{(\rho_0,0,0,0,0)}d\rho_0=8|\rho_0|^2d\rho_0\,.
$$




\end{document}